\documentclass[11pt,reqno]{amsart}
\usepackage[left=3.25cm,right=3.25cm,top=2.5cm,bottom=2.5cm,includeheadfoot]{geometry}
\usepackage[utf8]{inputenc}
\usepackage[T1]{fontenc} 

\usepackage{amsmath,amssymb,amsthm}
\usepackage{amstext,amsfonts}
\usepackage{enumerate}
\usepackage{enumitem}

\usepackage{hyperref}

\DeclareMathOperator{\Span}{span}
\DeclareMathOperator{\Div}{div}

\theoremstyle{plain}
\newtheorem{definition}{Definition}[section]
\newtheorem{theorem}[definition]{Theorem}
\newtheorem{lemma}[definition]{Lemma}
\newtheorem{corollary}[definition]{Corollary}

\theoremstyle{definition}
\newtheorem*{remark}{Remark}
\newtheorem*{example}{Example}

\title[A restricted superposition principle for (non-)linear FPKEs on HS]{A restricted superposition principle for (non-)linear Fokker--Planck--Kolmogorov equations on Hilbert spaces}
\author{Martin Dieckmann}
\email{martin.dieckmann@math.uni-bielefeld.de}
\address{Bielefeld University, Universitätsstrasse 25, 33615 Bielefeld, Germany}

\begin{document}
\begin{abstract}
We prove a version of the Ambrosio--Figalli--Trevisan superposition principle for a restricted subclass of solutions to the Fokker--Planck--Kolmogorov equation, that is valid on separable infinite-dimensional Hilbert spaces. 
Furthermore, we transfer this restricted superposition principle into a nonlinear setting.
\end{abstract}

\maketitle

\section{Introduction}
In \cite[Theorem 2.5, p.\ 7]{Tre16}, D.~Trevisan proved a superposition principle for linear Fokker--Planck--Kolmogorov equations (short: FPKEs) on $\mathbb R^n$, extending the prior seminal work of L.~Ambrosio (see \cite{Amb08}) and A.~Figalli (see \cite{Fig08}) substantially, assuming only fairly weak regularity and integrability conditions on the coefficients.
An infinite-dimensional analogue on $\mathbb R^\infty$, equipped with the product topology (i.e.~solutions of the martingale problem are probability measures on $C([0,T];\mathbb R^\infty)$), can be found in \cite[Section 7.1]{Tre14}.

This article will focus on the Hilbert space case equipped with its norm topology, which is in itself a very different approach to the problem.
In general, we have to impose stricter (but still commonly used) compactness assumptions to ensure that the constructed martingale solutions are supported on a path space with values in a separable Hilbert space and where the paths are continuous with respect to the norm topology of e.g.~another, possibly larger separable Hilbert space instead of on $C([0,T];\mathbb R^\infty)$ with its componentwise continuity.

First of all, this direction for a generalization is interesting because many applications typically have their setting in Hilbert spaces for which the superposition principle on $\mathbb R^\infty$ is insufficient.
Furthermore, the connection between probability solutions to FPKEs (see Definition \ref{def:linear-FPKE-notion-solution} below) and martingale solutions in the original sense of Stroock--Varadhan (see \cite{SV79}) to the corresponding martingale problem via the superposition principle is of most scientific value in a setting without uniqueness.
Having uniqueness, the probability solution could directly be generated from the martingale solution through its time-marginal laws by simply setting $\mu_t := P \circ x(t)^{-1}$ and would, of course, coincide with any constructed solution to the FPKE.
A non-unique Hilbert space setting is most prominently the case for $d$-dimensional stochastic Navier--Stokes equations (short: SNSEs) making it a prime candidate for an application of the methods studied.

Before proceeding, we should point out that in general the superposition principle on a separable Hilbert space $\mathbb H$ does not always hold, allowing us to realize that a simple and direct adaption of Trevisan's result is not to be expected.
\begin{example}
\item Let $\mathbb H = \ell_2$ and let the Kolmogorov operator $L$ be given by the generator of an Ornstein--Uhlenbeck process in $\ell_2$ with a constant diffusion coefficient and a (unbounded) linear drift $b$, that satisfies Fernique's necessary and sufficient condition (guaranteeing the non-continuity of sample paths, see \cite[Theorem 1.3.2, p.\ 11]{Fer75} for the original theorem).
Then the corresponding FPKE even has a stationary solution $\mu_t=\mu$ for every $t \geq 0$ with $\mu \in \mathcal P(\ell_2)$, but there exists no probability measure $P$ on $C([0,\infty);\ell_2)$ with time-marginals equal to $\mu$ for every $t$.
For additional details we refer to \cite[Example 6.6 (ii), p.\ 270f]{BR95}.
\end{example}

Hence, we can at most expect to derive a somehow restricted version of the superposition principle. 

In the first part of this work, which is devoted to the linear case (i.e.~Sections \ref{section:framework}--\ref{section:proof-of-thm1}), we begin by stating the first main result of this paper (see Theorem \ref{cor:SP-onH-linear-anySolutionwith} below in Section \ref{section:consequences}), which is a superposition principle on a separable Hilbert space $\mathbb H$ but only for a restricted subclass of solutions to the FPKE. 
In short, this means that for any given probability solution $(\mu_t)_{t \in [0,T]}$ to an infinite-dimensional FPKE, for which there already exists a subsequence of finite-dimensional solutions being created by Galerkin approximations and converging weakly to $(\mu_t)_{t \in [0,T]}$ as well as the necessary integrability conditions and assumptions for the corresponding martingale problem, we immediately obtain a martingale solution $P$ to the infinite-dimensional martingale problem satisfying the relation
\begin{align}\label{eq:intro-SP-eq}\tag{SP}
  P \circ x(t)^{-1} = \mu_t
\end{align}
for every $t \in [0,T]$, i.e.~$P$ ``superposes'' the given solution of the FPKE.

The proof of Theorem \ref{cor:SP-onH-linear-anySolutionwith} follows directly from the more general construction of a crucial lemma (see Lemma \ref{thm:SP-onH-linear} below) that we introduce in Section \ref{section:auxiliary}.
In this lemma we will take a step back and start with combining and adapting well-known approaches from the literature to ensure ``joint'' existence of our desired solutions.
To be more precise, we use that, on the one hand, the authors of \cite{BDRS15} (see also \cite[Section 10.4]{BKRS15}) construct a probability solution to an infinite-dimensional FPKE as a weak limit of finite-dimensional solutions.
On the other hand, in the article \cite{GRZ09} (see also \cite{RZZ15} for a more refined proof based on the same techniques in a setting with delay), the authors construct a martingale solution to a corresponding infinite-dimensional martingale problem on $\mathbb H$ in a very similar fashion using Galerkin approximations.
This way, we will in particular answer the question if and how those constructions are linked.

In fact, with our approach we still closely follow the fundamental question answered by the superposition principle in finite dimensions, namely: In a non-unique setting, can we obtain a martingale solution on the path space from given time-marginals, which necessarily have to solve the corresponding FPKE on $\mathbb H$?
Motivated by that question, we will first focus on our plan to construct a probability solution $(\mu_t)_{t\in[0,T]}$ to the FPKE 
\begin{align*}
  \partial_t \mu_t = L^{*}  \mu_t
\end{align*}
(the same way as in \cite{BDRS15} and \cite{BKRS15}) as a weak limit of finite-dimensional probability solutions $(\mu_{t,n})_{t\in[0,T]}$, $n\in \mathbb N$, for a suitable approximating FPKE on $\mathbb R^n$ and generate the corresponding finite-dimensional martingale solutions $P_n$ through the time-marginals via the finite-dimensional superposition principle.
The martingale solution $P$ which is obtained as a weak limit of those $P_n$ (by using the same methods as in \cite{GRZ09} and \cite{RZZ15}) is, therefore, still generated in the sense of the superposition principle. 
In addition, our construction fulfills the infinite-dimensional Equation \eqref{eq:intro-SP-eq} in the limit as well as all estimates from both individual constructions.

In Section \ref{section:navier-stokes} we will discuss an application of Theorem \ref{cor:SP-onH-linear-anySolutionwith} and Lemma \ref{thm:SP-onH-linear} to SNSEs.
Let us note that all four references \cite{GRZ09}, \cite{RZZ15}, \cite{BDRS15}, \cite{BKRS15} feature SNSEs as an application making it an obvious choice also in our case.

The second part of this work (i.e.~Sections \ref{section:framework-nonlinear} and \ref{section:nonlinear-results}) is devoted to the nonlinear case in which we will adapt Theorem \ref{cor:SP-onH-linear-anySolutionwith} to a nonlinear version of the restricted superposition principle on $\mathbb H$ (see Theorem \ref{thm:SP-nonlinear-onH} below).
We make use of the idea in \cite{BR18a} and \cite[Section 2]{BR18b} on ``freezing'' a nonlinear solution.
This means, that if we are given a probability solution $(\mu_t)_{t \in [0,T]}$ to a nonlinear FPKE
\begin{align*}
  \partial_t \mu_t = L_{\mu_t}^{*} \, \mu_t,
\end{align*}
we fix this $(\mu_t)_{t \in [0,T]}$ and consider the linear FPKE
\begin{align*}
  \partial_t \varrho_t = L_{\mu_t}^{*} \, \varrho_t
\end{align*}
for which $(\mu_t)_{t \in [0,T]}$ is a particular solution.
This allows us to apply results for linear FPKEs, but now with coefficients depending on this fixed path of measures $(\mu_t)_{t \in [0,T]}$. 
In our case, we will assume that the assumptions on our coefficients are uniform in the measure-component (see Subsection \ref{section:nonlinear-Assumptions} below) and, hence, satisfy all assumptions necessary for Theorem \ref{cor:SP-onH-linear-anySolutionwith}.
The martingale solution that we obtain for the martingale problem with coefficients $(t,x)\longmapsto b(t, x, \mu_{t})$ and $(t,x)\longmapsto\sigma(t, x, \mu_{t})$ is connected to a corresponding McKean--Vlasov SDE (see \cite{BR18a}, \cite{BR18b} for details).

\section{Framework}\label{section:framework}
First, let us introduce the framework obtained by carefully combining both settings from \cite{GRZ09}, \cite{RZZ15} (in the simplified case where $\mathbb Y = \mathbb H$ and on a time interval $[0,T]$ instead of $[0, \infty)$) and \cite{BDRS15}.

Let $T > 0$ and let $\mathbb H$ be a separable Hilbert space with inner product $\langle \cdot, \cdot \rangle_{\mathbb H}$ and norm $\| \cdot \|_{\mathbb H}$.
Recall that all infinite-dimensional separable Hilbert spaces are isometrically isomorphic to $\ell^2$ (see e.g.~\cite[Remark 10, p.\ 144]{Bre11}) and that we can treat $\ell^2$ as a subspace of $\mathbb R^\infty$, where $\mathbb R^\infty$, equipped with the product topology, is a Polish space.
This means we consider the continuous and dense embedding
\begin{align*}
  \ell^2 \subseteq \mathbb R^\infty
\end{align*}
as it is done in \cite{BDRS15}.
Let $\{ e_1, e_2, \dots \}$ be the standard orthonormal basis in $\ell^2$.
Then we define $\mathbb H_n := \Span\{e_1, \dots, e_n \}$, for $n \in \mathbb N$.

To study martingale problems on $\mathbb H$, we introduce another separable Hilbert space $\mathbb X$ for which the embedding
\begin{align*}
  \mathbb X \subseteq \mathbb H \simeq \mathbb H^* \subseteq \mathbb X^{*}
\end{align*}
is continuous, dense and compact.
Here, we write $\mathbb X^{*}$ for the dual space of $\mathbb X$ and $\| \cdot \|_{\mathbb X^{*}}$, $\| \cdot \|_{\mathbb X}$ denote their respective norms.
In order to apply the methods of \cite{GRZ09}, we have to ensure that $\{ e_1, e_2, \dots \} \subseteq \mathbb X$ holds and we have $\| \Pi_n z \|_{\mathbb X^{*}} \leq \| z \|_{\mathbb X^{*}}$ for every $z \in \mathbb X^*$.
Here, the projection $\Pi_n \colon \mathbb X^* \longrightarrow \mathbb H_n$ is defined by 
\begin{align*}
  \Pi_n z := \sum_{i=1}^n {}_{\mathbb X^*}^{}\langle z, e_i \rangle_{\mathbb X}^{} \, e_i, \quad z \in \mathbb X^*.
\end{align*}

Therefore, we follow \cite[Proposition 3.5, p.\ 424]{AR89} and \cite[Remark 3, p.\ 136f]{Bre11} (in particular using that $\mathbb X \subseteq \mathbb H$ is compact) and identify $\mathbb X$ with the weighted $\ell^2$-space $\ell^2( \lambda_i)$ for some sequence $(\lambda_i)_{i \in \mathbb N}$ with $\lim\limits_{i \rightarrow \infty} \lambda_i = \infty$ and $\lambda_i \geq 0$.
By considering its dual $\ell^2 \big(  \frac{1}{\lambda_i}  \big)$ we arrive at the embedding
\begin{align*}
 \phantom{\mathbb R^\infty_0 \subseteq}  \ell^2( \lambda_i) \subseteq  \ell^2 \subseteq \ell^2 \big(  \tfrac{1}{\lambda_i}  \big) \subseteq \mathbb R^\infty,
\end{align*}
where the dual pairing between $\ell^2( \lambda_i)$ and $\ell^2 \big(  \frac{1}{\lambda_i}  \big)$ is given by
\begin{align*}
  {}_{\mathbb X^*}^{}\langle z, v \rangle_{\mathbb X}^{} = \sum\limits_{i=1}^\infty z^i v^i,
\end{align*}
for any $z \in \mathbb X^*$ and $v \in \mathbb X$.

\begin{remark}
It follows from Kuratowski's theorem (see e.g.~\cite[p.\ 487f]{Kur66} or \cite[Section I.3,  p.\ 15ff]{Par67}) that we have $\mathbb X \in \mathcal B ( \mathbb H )$, $\mathbb H \in \mathcal B (\mathbb X^*)$, $\mathbb X^* \in \mathcal B(\mathbb R^\infty)$ and $\mathcal B ( \mathbb X )= \mathcal B ( \mathbb H ) \cap \mathbb X $, $\mathcal B ( \mathbb H )= \mathcal B ( \mathbb X^* ) \cap \mathbb H $, $\mathcal B ( \mathbb X^* )= \mathcal B ( \mathbb R^\infty ) \cap \mathbb X^*$.
Here, for a topological space $E$, $\mathcal B(E)$ denotes the Borel-$\sigma$-algebra and, for a subset $U \subseteq E$, the expression $\mathcal B(E) \cap U := \{ B \cap U \mid B \in \mathcal B(E) \}$ means the trace-$\sigma$-algebra on $U$. 
\end{remark}

\begin{remark}
We see that the projection $\Pi_n$ onto $\mathbb H_n$ in $\mathbb X^*$ in fact simplifies to
\begin{align*}
  \Pi_n z = \sum_{i=1}^n {}_{\mathbb X^*}^{}\langle z, e_i \rangle_{\mathbb X}^{} \, e_i = \sum_{i=1}^n z^i e_i = (z^1, \ldots, z^n, 0, \ldots)
\end{align*}
for any $z \in \mathbb X^*$.
\end{remark}
In addition, let $\mathbb U$ be another separable Hilbert space with inner product $\langle \cdot, \cdot \rangle_{\mathbb U}$ and norm $\| \cdot \|_{\mathbb U}$.

\smallskip
\noindent
\textbf{Path-space:}
Let 
\begin{align*}
  \Omega := C \big( [0, T] ; \mathbb X^* \big)
\end{align*}
be the space of all continuous functions from $[0, T]$ to $\mathbb X^*$.
By $x \colon \Omega \longrightarrow \mathbb X^*$ we denote the canonical process on $\Omega$ given by $x(t,\omega):= \omega(t)$.
We define the $\sigma$-algebra
\begin{align*}
 \mathcal F_t := \sigma \big( x(s) \, \big| \, s \in [0,t] \big)
\end{align*}
for every $t \in [0,T]$. 
For every $n \in \mathbb N$ we set
\begin{align*}
  \Omega_{n}:=C([0,T];\mathbb{H}_n).
\end{align*} 
Furthermore, we denote by $x_{n}$ the canonical process on $\Omega_n$ given by $x_n(t,\omega):=\omega(t)$ and define
\begin{align*}
  \mathcal{F}_t^{(n)}:=\mathcal{B}(C([0,t];\mathbb{H}_n)).
\end{align*} 

\noindent
\textbf{Further spaces:}
Define
\begin{align*}
  \mathbb S := C \big( [0, T]; \mathbb X^* \big) \cap L^p ([0,T]; \mathbb H),
\end{align*}
where the $p \geq 2$ is later to be specified in our assumptions (see Subsection \ref{subsection:linear-assumptions} below).
Note that $\mathbb S$ is a Polish space.

We define the following classes of so-called finitely based functions given by
\begin{align}\label{def:FC-c-infty-functions}
\hspace{-0.5em}
\begin{split}
     \mathcal FC^{2} (\{e_i\}) &:=  \big\{ f \colon \mathbb R^\infty \longrightarrow \mathbb R \, \big| \, f(y) = g\big( y^1, \ldots, y^d \big), d\in \mathbb N, g \in C^2(\mathbb R^d) \big\},
  \\ \mathcal FC^\infty_c (\{e_i\}) &:= \big\{ f \colon \mathbb R^\infty \longrightarrow \mathbb R \, \big| \, f(y) = g\big( y^1, \ldots, y^d \big), d\in \mathbb N, g \in C^\infty_c(\mathbb R^d) \big\}
\end{split}
\end{align}
(see e.g.~\cite[p.\ 54]{MR92} or \cite[p.\ 404f]{BKRS15}).

For two separable Hilbert spaces $\mathbb H_1$ and $\mathbb H_2$, let $L_2 \big(\mathbb H_1; \mathbb H_2 \big)$ be the space of all Hilbert--Schmidt operators from $\mathbb H_1$ to $\mathbb H_2$ with norm $\| \cdot \|_{L_2 ( \mathbb H_1; \mathbb H_2 )}$.

By $\mathfrak U^\varrho$, for $\varrho \geq 1$, we denote the class of functions $\mathcal N \colon \mathbb H \longrightarrow [0, \infty]$ with the following properties:
\begin{enumerate}[label=(\roman*), leftmargin=40pt]
\item $\mathcal N (y)=0$ implies $y=0$, 
\item $\mathcal N (c y) \leq c^\varrho \mathcal N(y)$ holds for every $c \geq 0$ and $y \in \mathbb H$, 
\item the set $  \{ y \in \mathbb H \, | \, \mathcal N(y) \leq 1 \}$ is compact in $\mathbb H$. 
\end{enumerate}
\begin{remark} 
From properties (i)--(iii) we can conclude that any function in $\mathfrak U^\varrho$ is lower semi-continuous on $\mathbb H$.
Furthermore, we can extend a function $\mathcal N \in \mathfrak U^\varrho$ to a $\mathcal B ( \mathbb X^*)/ \mathcal B ([0, \infty])$-measurable one on $\mathbb X^*$ by setting $\mathcal N(y)=\infty$ for $y \in \mathbb X^* \setminus \mathbb H$.
Note that $\mathcal N$, as a function on $\mathbb X^*$, is still lower semi-continuous since the embedding $\mathbb H \subseteq\mathbb X^*$ is continuous and compact.
\end{remark}

\noindent
\textbf{Coefficients:}
Let the mappings
\begin{align*}
  \sigma &\colon  [0,T] \times \mathbb H \longrightarrow L_2(\mathbb U; \mathbb H),
 \\ b &\colon   [0,T] \times \mathbb H \longrightarrow \mathbb X^{*}
\end{align*}
be Borel-measurable.
In order to obtain components of those coefficients that are defined on $[0,T] \times \mathbb R^\infty$ for the FPKE, we just extend $b$ and $\sigma$ by $0$ on $\mathbb R^\infty \setminus \mathbb H$.
This means, for every $i,j \in \mathbb N$, we consider the $\mathcal B ([0,T]) \otimes \mathcal B(\mathbb R^\infty) / \mathcal B(\mathbb R)$-measurable mappings
\begin{align*}
a^{ij} &\colon  [0,T] \times \mathbb R^\infty \longrightarrow \mathbb R,\\
b^{i} &\colon   [0,T] \times \mathbb R^\infty \longrightarrow \mathbb R,
\end{align*}
that are given by
\begin{align*}
  a^{ij}(t,y) := \begin{cases}  \frac12 \langle \sigma(t,y) \sigma(t,y)^* e_i,  e_j \rangle_{\mathbb H} ,  &(t,y) \in [0,T] \times \mathbb H, \\ 0,   &(t,y) \in [0,T] \times \mathbb R^\infty \setminus \mathbb H \end{cases}
\end{align*}
and
\begin{align*}
b^{i}(t,y) := \begin{cases} {}_{\mathbb X^*}\langle b(t,y), e_i \rangle_{\mathbb X} ,  &(t,y) \in [0,T] \times \mathbb H, \\ 0,   &(t,y) \in [0,T] \times \mathbb R^\infty \setminus \mathbb H. \end{cases}
\end{align*}
In addition, we set 
\begin{align}\label{eq:lFPKE:b_n-A_n}
 b_n:= (b^1, \ldots, b^n) \text{ and } A_n:=(a^{ij})_{1 \leq i,j \leq n}
\end{align}
as well as
\begin{align*}
 A:=(a^{ij})_{1 \leq i,j < \infty}.
\end{align*}

\begin{remark}
We note that $a^{ij}$ and $b^i$, regardless of our choice to simply extend them by $0$ on $\mathbb R^\infty \setminus \mathbb H$, will still be admissible mappings to satisfy all necessary assumption from Subsection \ref{subsection:linear-assumptions} below, because those assumptions are either imposed on $\mathbb H_n$ anyway or remain unchanged as e.g.~symmetry or growth.
\end{remark}

\subsection{Equation}
Let $x_0 \in \mathbb H$ and denote by $\delta_{x_0}$ the Dirac measure at this point.
We will, for simplicity, reduce our calculations to this choice of an initial measure, but note that by integrating over all these measures, we can generalize our results, since the FPKE depends linearly on the initial condition.
Note that this will not be possible anymore in the nonlinear case in Sections \ref{section:framework-nonlinear} and \ref{section:nonlinear-results}.

\smallskip
\noindent
\textbf{Kolmogorov operator and FPKE:}
Let $L$ be the Kolmogorov operator, acting on finitely based functions $\varphi \in \mathcal FC^2( \{e_i\})$, which is given by
\begin{align*}
  L \varphi (t,y) = \sum\limits_{i,j= 1}^{\infty} a^{ij}(t,y) \partial_{e_i} \partial_{e_j} \varphi(y) + \sum\limits_{i= 1}^{\infty} b^{i}(t,y) \partial_{e_i} \varphi(y),
\end{align*}
for $(t,y) \in [0,T] \times \mathbb R^\infty$.
\begin{remark}
As usual, $L$ obviously also acts on finitely based functions that are in addition explicitly depending on time, because this time-dependence is ``irrelevant'' for the partial derivatives appearing in the operator.
But we will not need the often used classes of time-dependent test functions and rather mostly apply $L$ to functions $\varphi \in \mathcal FC_{c}^\infty(\{e_i\})$ in the following.
\end{remark}

Consider the following shorthand notation for an infinite-dimensional linear Fokker--Planck--Kolmogorov equation associated to the Kolmogorov operator $L$, which is a Cauchy problem of the form 
\begin{align}\label{eq:FPKE-linear}\tag{FPKE}
\begin{split}
\partial_t \mu_t &= L^* \mu_t,
\\ \mu_{0} &= \delta_{x_0},
\end{split}
\end{align}
where $L^*$ is the formal adjoint of $L$ and the solution $(\mu_t)_{t \in [0,T]}$ is a path of Borel probability measures on $\mathbb R^\infty$.

\subsection{Notion of solution}
Let us introduce the notion of a probability solution to Equation \eqref{eq:FPKE-linear} and the notion of a martingale solution to the martingale problem associated to the same operator $L$ in the sense of Stroock--Varadhan.

\begin{definition}[probability solution]\label{def:linear-FPKE-notion-solution}
A path $(\mu_t)_{t \in [0,T]}$ of Borel probability measures on $\mathbb R^\infty$ is called probability solution to Equation \eqref{eq:FPKE-linear} if the following conditions hold.
\begin{enumerate}[label=(\roman*), leftmargin=30pt]
 \item\label{condition:probability-solution-i-L1} The functions $a^{ij}$, $b^{i}$ are integrable with respect to the measure $\mu_t \, \mathrm{d}t$, i.e.
 \begin{align*}
    a^{ij}, b^{i} \in L^1([0,T]  \times \mathbb R^\infty, \mu_t \, \mathrm{d}t).
 \end{align*}
 \item\label{condition:probability-solution-ii-integral}
 For every function $\varphi\in \mathcal FC_c^\infty ( \{e_i\} )$ we have
\begin{align}\label{eq:lFPKE-solution-equivalent-no-time-dependence}
    \int_{\mathbb R^\infty}  \varphi(y) \, \mu_t(\mathrm{d}y)=  \int_{\mathbb R^\infty} \varphi(y) \, \delta_{x_0}(\mathrm{d}y) + \int_0^{t} \int_{\mathbb R^\infty}  L \varphi(s,y)  \,  \mu_s(\mathrm{d}y) \, \mathrm{d}s
\end{align}
for $\mathrm{d}t$-a.e. $t \in [0,T]$.
\end{enumerate}
\end{definition}

\begin{definition}[martingale solution] \label{def:martingale-solution}
A probability measure $P \in \mathcal P (\Omega)$ is called martingale solution to the martingale problem with coefficients $b$ and $\sigma$ and initial value $x_0 \in \mathbb H$ if the following conditions hold.
\begin{enumerate}[label=\textbf{(M\arabic*)}, leftmargin=40pt]
 \item \label{Bedingung-1-martingale-solution} $P \big[ x(0)=x_0 \big] = 1$ and
 \\ $P \Big[ x \in \Omega \, \Big| \, \text{For } \mathrm{d}s\text{-a.e.~} s\in[0,T]: x(s)\in \mathbb H \,  \text{ and }$
  \\ $~\hspace*{3em}\displaystyle\int_0^T \big\| b(s, x(s) \big) \big\|_{\mathbb X^*} \, \mathrm{d} s + \displaystyle\int_0^T \big\| \sigma (s, x(s)\big) \big\|^2_{L_2 (\mathbb U ; \mathbb H)} \, \mathrm{d} s < \infty \Big]=1$.
 \item \label{Bedingung-2-martingale-solution} For every function $f \in \mathcal FC_c^\infty( \{e_i \})$ the process
\begin{align*}
  \mathbb M^f(t,x):= f(x(t)) - f(x_0) - \int_0^t Lf(s,x(s)) \, \mathrm{d}s, \quad t \in [0,T],
\end{align*}
is an $(\mathcal F_t)$-martingale with respect to $P$.
\end{enumerate}
\end{definition}

It is important to note that in \cite[Definition 3.1, p.\ 1730]{GRZ09} Condition (M2) in the notion of a martingale solution is stated in a ``weak formulation'' involving the inner product.
In fact, by using Itô's formula we can directly show that this implies our Condition \ref{Bedingung-2-martingale-solution} used in Definition \ref{def:martingale-solution}.
Furthermore, we can completely drop Condition (M3) from \cite{GRZ09} as a requirement for being a solution because we can simply transform it into an a priori energy estimate as already done in \cite[Lemma 3.1, p.\ 368]{RZZ15}.

Let us now state what we exactly mean by a martingale problem in the original sense of Stroock--Varadhan arising from given coefficients $b$ and $\sigma$ and an initial value $x_0 \in \mathbb H$.
In concrete terms, we will consider the following problem:
\begin{align}\label{eq:martingale-problem}\tag{MP}
\begin{tabular}{c}
 \text{Existence of a martingale solution $P \in \mathcal P(\mathbb S)$ in the sense of}\\
 \text{Definition \ref{def:martingale-solution} for coefficients $b$ and $\sigma$ and with initial}\\
 \text{value $x_0 \in \mathbb H$},
\end{tabular}
\end{align}
where $P \in \mathcal P(\mathbb S)$ means that we are explicitly searching for solutions that also require paths from the path space $C \big( [0, T]; \mathbb X^* \big)$ to be of class $L^p ([0, T]; \mathbb H)$.

\begin{remark}
We would like to stress that we do not focus on the, in the literature often naturally established, link to weak solutions of infinite-dimensional stochastic differential equations (see e.g.~\cite[Theorem 2.2, p.\ 364]{RZZ15}, which is substantially using \cite[Theorem 2, p.\ 1007]{Ond05}), and only consider martingale problems in the original sense of Stroock--Varadhan.
\end{remark}

\subsection{Assumptions}\label{subsection:linear-assumptions}
Before stating the necessary assumptions on the coefficients, let us quickly recall the meaning of a compact and a non-degenerate function.
\begin{definition}[compact function, see e.g.~{\cite[Definition 2.3.1, p.\ 62]{BKRS15}}]
A real-valued function $f$ on a topological space is called compact if the sublevel sets $\{ f \leq R \}$ are compact for any $R \in \mathbb R$.
\end{definition}

\begin{definition}[non-degenerate function, see e.g.~{\cite[p.\ 410]{BDR08-CPDE}}]
A compact function $f \in C^2(\mathbb R^n)$ is called non-degenerate if there exists a sequence $(c_k)_{k \in \mathbb N}$ of numbers with $c_k \xrightarrow[k \rightarrow \infty]{} \infty$ such that the level sets $f^{-1}(c_k)= \{ y \in \mathbb R^n \mid f(y) = c_k \}$ are $C^1$-surfaces.
\end{definition}

The following assumptions on our coefficients are, up to some minor modifications to ensure applicability of the used finite-dimensional results, directly taken from our main references.
Note that Assumptions \ref{Assumption:lFPKE-Theta}--\ref{Assumption:lFPKE-Coefficients} will only be used for Lemma \ref{thm:SP-onH-linear} (see Section \ref{section:auxiliary} below), but are already listed here for completeness.

\begin{enumerate}[label=\textbf{(H\arabic*)}, leftmargin=40pt]
\item \label{Assumption:lFPKE-sym-nonneg} For all $n \in \mathbb N$, the matrices $A_n= (a^{ij})_{1 \leq i,j \leq n}$ are symmetric and nonnegative definite.
\item \label{Assumption:lFPKE-Theta} Let $\Theta \colon \mathbb R^\infty \longrightarrow [0, \infty]$ be a compact Borel function, bounded on bounded sets on each space $\mathbb H_n$, $n \in \mathbb N$, such that, for every $i \in \mathbb N$ and $j \leq i$,
\begin{itemize}
\item[$\bullet$] the functions $y \mapsto a^{ij}(t,y)$, $t \in [0,T]$, are equicontinuous on every set $\{\Theta\le R\}$ with $R < \infty$ and also on every fixed ball in each $\mathbb H_n$,
\item[$\bullet$] for every $t \in [0,T]$ the function $y \mapsto b^{i}(t,y)$ is continuous on every set $\{\Theta\le R\}$ with $R < \infty$ and also on each $\mathbb H_n$.
\end{itemize}
\item \label{Assumption:lFPKE-V} There exist numbers $M_0,C_0\geq 0$ and a compact Borel function $V \colon \mathbb R^\infty \longrightarrow [1,\infty]$ whose restrictions to $\mathbb H_n$ are of class $C^2(\mathbb H_n)$ and non-degenerate such that for all $y\in \mathbb H_n$, $n \in \mathbb N$, $t \in [0,T]$,  we have
\begin{align*}
  \sum\limits_{i,j=1}^n a^{ij}(t,y)\partial_{e_i}V(y)\partial_{e_j}V(y) &\leq M_0 V(y)^2, \\
  LV(t,y)&\leq C_0 V(y)-\Theta(y).
\end{align*}
\item \label{Assumption:lFPKE-Coefficients}  There exist constants $C_i \geq 0$ and $k_i \geq 0$ such that for all $i \in \mathbb N$ and $j \leq i$ we have
\begin{align*}
  |a^{ij}(t,y)|+|b^i(t,y)| \leq C_i V(y)^{k_i} \big( 1+\kappa_i(\Theta(y))\Theta(y) \big),
\end{align*}
for every $(t,y) \in [0,T] \times \mathbb R^\infty$, where $\kappa_i$ is a bounded nonnegative Borel function on $[0,\infty)$ with $\lim\limits_{s\to \infty}\kappa_i(s)=0$.
\end{enumerate}

\begin{enumerate}[label=\textbf{(N)}, leftmargin=40pt]
  \item \label{Assumption:N1-MP-SDE} There exists a function $\mathcal N \in \mathfrak U^p$ for some $p \geq 2$ such that for every $n \in \mathbb N$ there exists a constant $C_n \geq 0$ with
      \begin{align*}
      \mathcal N(v) \leq C_n \| v \|^p_{\mathbb H_n},
    \end{align*}
    for any $v \in \mathbb H_n$.
\end{enumerate}

\begin{enumerate}[label=\textbf{(A\arabic*)}, leftmargin=40pt]
 \item (Demicontinuity)\label{Assumption:A1-Demicontinuity-SDE} For any $v \in \mathbb X$, $t\in[0,T]$ and every sequence $(y_k)_{k \in \mathbb N}$ with $y_k \xrightarrow[k \rightarrow \infty]{} y$ in $\mathbb H$, we have 
 \begin{align*}
 \lim\limits_{k \rightarrow \infty} \, _{\mathbb X^*}\langle b(t, y_k),v \rangle_{\mathbb X} = \, _{\mathbb X^*}\langle b(t, y),v \rangle_{\mathbb X}
 \end{align*}
  and
 \begin{align*}
 \lim\limits_{k \rightarrow \infty} \big\| \sigma^* (t, y_k \big) (v) -  \sigma^* (t, y \big) (v) \big\|_{\mathbb U} = 0.
 \end{align*}
\item (Coercivity)\label{Assumption:A2-Coercivity-SDE} There exists a constant $\lambda_1 \geq 0$ such that for all $v \in \mathbb X$ and $t\in[0,T]$
 \begin{align*}
  _{\mathbb X^*}\langle b(t, v),v \rangle_{\mathbb X} \leq - \mathcal N(v) + \lambda_1 (1+ \|v\|^2_{\mathbb H})
  \end{align*}
holds.
\item (Growth)\label{Assumption:A3-Growth-SDE} There exist constants $\lambda_2, \lambda_3, \lambda_4 > 0$ and constants $\gamma' \geq \gamma > 1$ such that for all $y \in \mathbb H$ and $t\in[0,T]$ we have
  \begin{align*}
    \|b(t,y)\|^\gamma_{\mathbb X^*} &\leq \lambda_2 \, \mathcal N(y) + \lambda_3 (1+ \|y\|^{\gamma'}_{\mathbb H})
  \end{align*}
  and
  \begin{align*}
    \|\sigma(t,y)\|^2_{L_2 (\mathbb U ; \mathbb H)} &\leq \lambda_4 (1+ \|y\|^2_{\mathbb H}).
  \end{align*}
\end{enumerate}

Furthermore, in order to guarantee that our initial measure $\delta_{x_0}$ satisfies all condition assumed in \cite[Theorem 3.1, p.\ 1013]{BDRS15}, we assume that
\begin{align*}
  W_k := \sup_{n\in\mathbb N} \|V(\cdot)^k \circ \Pi_n\|_{L^1(\delta_{x_0})} = \sup_{n\in\mathbb N} | V^k(\Pi_n x_0) | <\infty 
\end{align*}
holds for all $k\in \mathbb{N}$.

\begin{remark}
We note that e.g.~in \cite[Proposition 7.1.8, p.\ 293]{BKRS15} we can find the idea for a transformation of a given Lyapunov function $V$ to one that already satisfies integrability with respect to the initial measure in the finite-dimensional setting.
Adapting this idea  would be an option to actually drop the above assumption on $W_k$.
\end{remark}

We also want to note that these assumptions are not supposed to be perfectly optimal and leave room for improvement and unification.
In particular, Assumption \ref{Assumption:A2-Coercivity-SDE} would be a prime candidate to be transformed into a Lyapunov condition similar to Assumption \ref{Assumption:lFPKE-V}.
But for a start, we impose the combination of both sets of respective assumptions, because we are confident that coefficients in potential applications like SNSE will satisfy them anyway.

\section{Restricted superposition principle for linear FPKEs}\label{section:consequences}

Before stating our restricted version of a superposition principle, let us first introduce the finite-dimensional martingale problem for the projected coefficients on $\mathbb H_n$, which we can state as
\begin{align}\label{eq:MP-Hn}\tag{$\mathrm{MP}_n$}
\hspace{-0.2em}
\begin{tabular}{c}
 \text{Existence of a martingale solution $P_n \in \mathcal P(\Omega_n)$ in the sense of}\\
 \text{Definition \ref{def:martingale-solution} for coefficients $\Pi_n b$ and $\Pi_n \sigma$ and with initial}\\
 \text{value $\Pi_n x_0 \in \mathbb H_n$}
\end{tabular}
\end{align}
and the finite-dimensional FPKE on $\mathbb H_n$, that can be written in short-hand notation as
\begin{align}\label{eq:FPKE-linear-Hn}\tag{$\mathrm{FPKE}_n$}
\begin{split}
\partial_t \mu_{t,n} &= L_n^* \, \mu_{t,n},
\\ \mu_{0,n} &= \delta_{x_0} \circ \Pi_n^{-1}.
\end{split}
\end{align}
Here, the operator $L_n$, acting on functions $\varphi \in C^{2} (\mathbb H_n)$, is given by
\begin{align*}
  L_n \varphi (t,y) = \sum\limits_{i,j= 1}^{n} a^{ij}(t,y) \partial_{e_i} \partial_{e_j} \varphi(y) + \sum\limits_{i= 1}^{n} b^{i}(t,y) \partial_{e_i} \varphi(y),
\end{align*}
for $(t,y) \in [0,T] \times \mathbb H_n$.

\bigskip
The following theorem is the name-giving \textit{restricted superposition principle} and the main result for linear FPKEs.
It is in fact a direct consequence of the proof of Lemma \ref{thm:SP-onH-linear} (see Section \ref{section:auxiliary} for the lemma and Section \ref{section:proof-of-thm1} for the proof below).
By upfront assuming existence of a probability solution to \eqref{eq:FPKE-linear} as well as all desired properties for it, we do not have to impose Assumptions \ref{Assumption:lFPKE-Theta}--\ref{Assumption:lFPKE-Coefficients}.
Instead, we only ensure the application of the finite-dimensional superposition principle directly by assumption, where in particular Assumption \ref{Assumption:lFPKE-sym-nonneg} is a part of.

\begin{theorem}\label{cor:SP-onH-linear-anySolutionwith}
 Let Assumptions \ref{Assumption:N1-MP-SDE}, \ref{Assumption:A1-Demicontinuity-SDE}, \ref{Assumption:A2-Coercivity-SDE}, \ref{Assumption:A3-Growth-SDE} and \ref{Assumption:lFPKE-sym-nonneg} be fulfilled.
 Assume there exists a probability solution $(\mu_t)_{t \in [0,T]}$ on $\mathbb H$ to Equation \eqref{eq:FPKE-linear} in the sense of Definition \ref{def:linear-FPKE-notion-solution} and a subsequence $\{ ( \mu_{t,n_k} )_{t \in [0,T]} \mid  {k \in \mathbb N} \}$ of paths of Borel probability measures on $\mathbb H_{n_k}$ with the following properties:
 \begin{itemize}
    \item[$\bullet$] The paths $(\mu_{t,n_k})_{t \in [0,T]}$, $k \in \mathbb N$, are solutions to the finite-dimensional Equations \eqref{eq:FPKE-linear-Hn} on $\mathbb H_{n_k}$ with the property that the mapping
        \begin{align*}
          t\mapsto \int_{\mathbb H_{n_k}} \zeta(y)\, \mu_{t,n_k}(\mathrm{d}y)
        \end{align*}
        is continuous on $[0,T]$ for every $\zeta\in C_c^\infty(\mathbb H_{n_k})$.
    \item[$\bullet$] For the family $( \bar\mu_{t,{n_k}} )_{k \in \mathbb N}$ of extended measures to $\mathbb H$, we have $\bar \mu_{t, n_k} \xrightarrow[k \rightarrow \infty]{w} \mu_t$ for every $t \in [0,T]$.
    \item[$\bullet$] The integrability condition
    \begin{align*}
      \int_0^T \int_{\mathbb H_{n_k}} \frac{\| \Pi_{n_k} A(t,y) \Pi_{n_k}^* \| + | \langle \Pi_{n_k} b(t,y), y \rangle_{\mathbb H_{n_k}} |}{(1+\|y\|_{\mathbb H_{n_k}})^2} \, \mu_{t,_{n_k}}(\mathrm{d}y) \, \mathrm{d}t < \infty
    \end{align*}
    holds for every $k \in \mathbb N$.
\end{itemize}
Then there exists a martingale solution $P \in \mathcal P(\mathbb S)$ to the martingale problem \eqref{eq:martingale-problem} in the sense of Definition \ref{def:martingale-solution}, for which the time-marginal laws of $P$ coincide with $\mu_t$, i.e. 
\begin{align}\label{eq:SP-onH-marginal-law-property}
  P \circ x(t)^{-1} = \mu_{t}
\end{align}
holds for every $t \in [0,T]$.
\end{theorem}

\noindent
The proof of Theorem \ref{cor:SP-onH-linear-anySolutionwith} will be given below in Section \ref{section:proof-of-thm1}. 

\begin{remark}
 Let us note, that it is not sufficient to just restrict the measures $\mu_t$ to the finite-dimensional spaces $\mathbb H_n$, by e.g.~considering the push-forward measures $\mu_t \circ \Pi_n^{-1}$, in order to get a weakly convergent subsequence, because these measures not necessarily form a solution to the Equations \eqref{eq:FPKE-linear-Hn} with coefficients $\Pi_n b$ and $\Pi_n A \, \Pi_n^*$.
 We refer to \cite[Section 10.2, p.\ 413ff]{BKRS15} for more details on this kind of equation.
\end{remark}

The next corollary will highlight the fact that we can directly conclude continuity for the mapping $t \mapsto \mu_t$ with respect to the topology generated by finitely based functions from Equation \eqref{eq:SP-onH-marginal-law-property}.
\begin{corollary}\label{cor:solution-SP-weak-cont}
For solutions $P$ and $(\mu_t)_{t \in [0,T]}$ considered in Theorem \ref{cor:SP-onH-linear-anySolutionwith}, Equation \eqref{eq:SP-onH-marginal-law-property} implies that the mapping $t \mapsto \mu_t$ from $[0,T]$ to $\mathcal P(\mathbb H)$ is continuous with respect to the topology generated by the class $\mathcal FC_c^\infty( \{e_i \})$ of finitely based functions, i.e. that the mapping
\begin{align*}
  t \mapsto \int_{\mathbb H} f(y) \,  \mu_t(\mathrm{d} y)
\end{align*}
is continuous for every $f \in \mathcal FC_c^\infty( \{e_i \})$.
\end{corollary}

\begin{proof}
From Theorem \ref{cor:SP-onH-linear-anySolutionwith} we are given measures $\mu_t \in \mathcal P(\mathbb H)$, $t \in [0,T]$, and $P \in \mathcal P (\mathbb S)$ that satisfy Equation \eqref{eq:SP-onH-marginal-law-property}.

We know that for $P \in \mathcal P (\mathbb S)$ the canonical process $x$ on $\mathbb S$ is in particular a mapping in the path space $C \big( [0, T] ; \mathbb X^*  \big)$.
This means that for any $f \in \mathcal FC_c^\infty(\{e_i\})$, and with it some $d \in \mathbb N$ and $g \in  C_c^\infty(\mathbb R^d)$, the mapping
\begin{align*}
  t \longmapsto  g \big( {}_{\mathbb X^*}\langle x(t), e_1 \rangle_{\mathbb X} , \ldots, {}_{\mathbb X^*}\langle x(t), e_d \rangle_{\mathbb X} \big)
\end{align*}  
is continuous.
Hence, the mapping $t \mapsto \int_{\mathbb S} f(x(t)) \,  P(\mathrm{d} x)$ is continuous for any $f \in \mathcal FC_c^\infty(\{e_i\})$, which yields the assertion.
\end{proof}

\section{Auxiliary results}\label{section:auxiliary}
Let us state the mentioned crucial lemma, which can be described as a ``joint'' existence result for probability and martingale solutions which are connected through their time-marginal laws in the spirit of the superposition principle. 

\begin{lemma}\label{thm:SP-onH-linear}
Under the assumptions from Subsection \ref{subsection:linear-assumptions} there exists a probability solution $(\mu_t)_{t \in [0,T]}$ on $\mathbb H$ to Equation \eqref{eq:FPKE-linear} in the sense of Definition \ref{def:linear-FPKE-notion-solution} and a martingale solution $P \in \mathcal P (\mathbb S)$ to the associated martingale problem \eqref{eq:martingale-problem} in the sense of Definition \ref{def:martingale-solution}, for which Equation \eqref{eq:SP-onH-marginal-law-property} holds for every $t \in [0,T]$. 

In particular, the following estimates and equations hold.
For every $q \geq 1$, we have
 \begin{align}\label{eq:MP-result-estimate-solution}
  \mathbb E^P \bigg[\sup\limits_{t \in [0,T]} \|x(t)\|^{2q}_{\mathbb H} + \int_0^T \|x(t)\|^{2(q-1)}_{\mathbb H} \mathcal N(x(t)) \, \mathrm{d}t \bigg] < \infty.
 \end{align}

Furthermore, for all $t\in [0,T]$ and $k \in \mathbb N$, we have
\begin{align}\label{eq:lFPKE-estimate-sol-Wk}
\int_{\mathbb R^\infty} V(y)^k\,  \mu_t(\mathrm{d}y) + k \int_0^t \int_{\mathbb R^\infty} V(y)^{k-1} \Theta(y) \, \mu_s(\mathrm{d}y) \, \mathrm{d}s \leq N_k W_k,
\end{align}
where $N_k:=M_k e^{M_k}+1$ and $M_k:=k(C_0+(k-1)M_0)$,
as well as
\begin{align}\label{eq:lFPKE-V-finite-mu_t-1}
  \mu_t(V<\infty)=1
\end{align}
for all $t \in [0,T]$ and $\mu_t(\Theta<\infty)=1$ for $\mathrm{d}t$-a.e.~$t \in [0,T]$.

\end{lemma}

\noindent
The proof of Lemma \ref{thm:SP-onH-linear} will be given below in Section \ref{section:proof-of-thm1}. 

\begin{remark}
The statement of Corollary \ref{cor:solution-SP-weak-cont} remains valid in the setting of this lemma.
\end{remark}

\begin{remark}
Obviously, the respective assumptions from Subsection \ref{subsection:linear-assumptions} directly ensure existence for both martingale and probability solutions individually, but without any additional information (e.g.~on uniqueness) we a priori could not specify any such connection given in Equation \eqref{eq:SP-onH-marginal-law-property} while also preserving knowledge about both solutions, i.e.~Equation and Estimates \eqref{eq:MP-result-estimate-solution}--\eqref{eq:lFPKE-V-finite-mu_t-1}, gained in their construction as limits of finite-dimensional approximations.
\end{remark}

Let us stress, that our approach for this existence result implements our plan stated in the introduction to follow the main motivation behind the superposition principle.
In fact, we will see in Section \ref{section:proof-of-thm1} that we first construct $(\mu_t)_{t \in [0,T]}$ the same way as in \cite{BDRS15} and \cite{BKRS15} as a weak limit of finite-dimensional probability solutions $(\mu_{t,n})_{t\in[0,T]}$ for $n\in \mathbb N$ and generate the corresponding finite-dimensional measures $P_n$ through the time-marginals via the finite-dimensional superposition principle.
The martingale solution $P$ which is then obtained as a weak limit of those $P_n$ (by using the same methods as in \cite{GRZ09} and \cite{RZZ15}) is, therefore, generated in the sense of the superposition principle. 

In addition, we can derive from Equation \eqref{eq:MP-result-estimate-solution} even slightly more information about the probability solution than it was possible in \cite{BDRS15} and \cite{BKRS15}, since
\begin{align}
  \infty > \mathbb E^P \bigg[ \int_0^T \|x(t)\|^{2(q-1)}_{\mathbb H} \mathcal N(x(t)) \, \mathrm{d}t \bigg] = \int_0^T \|y\|^{2(q-1)}_{\mathbb H} \mathcal N(y) \, \mu_t(\mathrm{d}y)
\end{align}
holds by using the transformation theorem based on Equation \eqref{eq:SP-onH-marginal-law-property}.

Before actually proving Lemma \ref{thm:SP-onH-linear}, let us discuss the three well-known results that we will combine for it in the following.
We start with a short streamlined overview of the scheme of proof used in \cite{GRZ09} and \cite{BDRS15}, on the one hand, to recall it for the reader and, on the other hand, to show its similarity and to give a clear idea on how to make use of it.

\smallskip
\noindent
\textbf{Existence of martingale solution:} (Theorem 4.6 in \cite[p.\ 1739]{GRZ09})\\
First, the authors consider the finite-dimensional martingale problem on $\mathbb H_n$ with coefficients $\Pi_n b$ and $\Pi_n \sigma$ being created by the projections $\Pi_n$.
By using well-known results in finite dimensions (see \cite[Theorem 6.1.7, p.\ 144]{SV79}), they deduce existence of martingale solutions, i.e.~some $P_n \in \mathcal P(\Omega_n)$, for any $n \in \mathbb N$.
From there they extend $P_n$ to $\bar{P}_n \in \mathcal P(\Omega)$ and prove tightness of the family $(\bar{P}_n)_{n \in \mathbb N}$.
Then they extract a subsequence of $(\bar{P}_n)_{n \in \mathbb N}$ converging weakly to a probability measure $P \in  \mathcal P(\mathbb S)$ that is a solution to the infinite-dimensional martingale problem with coefficients $b$ and $\sigma$.

\smallskip
\noindent
\textbf{Existence of probability solution:} (Theorem 3.1 in \cite[p.\ 1013]{BDRS15})\\
First, the authors consider the finite-dimensional FPKE on $\mathbb H_n$ with coefficients $A_n$ and $b_n$, which consist of the components $a^{ij}$ and $b^i$ up to $n$ for any $n \in \mathbb N$.
They prove existence of solutions $(\mu_{t,n})_{t \in [0,T]}$ by using finite-dimensional results (see \cite[Corollary 3.4, p.\ 415]{BDR08-CPDE}).
Then, after extending the measures $\mu_{t,n}$ to $\bar \mu_{t,n}$ on $\mathbb R^\infty$ and proving tightness of the family $( \bar \mu_{t,n} )_{n \in \mathbb N}$, they extract a subsequence that is weakly converging to a probability measure $\mu_t$.
Finally, they prove that $(\mu_t)_{t \in [0,T]}$ is a probability solution to the infinite-dimensional FPKE with coefficients $A$ and $b$.

\smallskip
\noindent
\textbf{Finite-dimensional superposition principle:}
Let us quickly state the finite-dimensional superposition principle proved in \cite[Theorem 1.1, p.\ 5]{BRS19}, which is further weakening the integrability condition imposed in \cite{Tre16}.
This theorem is easier to use in our particular setting, making it our reference of choice for applying the superposition principle to probability solutions in finite dimensions later.
First, let us recall the necessary assumptions stated in \cite{BRS19}:
\begin{enumerate}[label=\textbf{(S\arabic*)}, leftmargin=40pt]
  \item \label{Assumption:SPP-S0-A-symmetric} The diffusion matrix $A_n=(a^{ij})_{1 \leq i,j \leq n}$ is symmetric and nonnegative definite.
  \item \label{Assumption:SPP-S1-coefficients-a-b} For every ball $U \subseteq \mathbb R^n$ we have
      \begin{align*}
        a^{ij}, b^{i} \in L^1([0,T] \times U, \mu_{t,n} \, \mathrm{d}t).
      \end{align*}  
  \item \label{Assumption:SPP-S2-integrability-a-b} The integrability condition
        \begin{align*}
          \int_0^T \int_{\mathbb R^n} \frac{\|A_n(t,y)\| + | \langle b_n(t,y), y \rangle_{\mathbb R^n} |}{(1+\|y\|_{\mathbb R^n})^2} \, \mu_{t,n}(\mathrm{d}y) \, \mathrm{d}t < \infty
        \end{align*}
       holds.
\end{enumerate}

Their finite-dimensional superposition principle can then be written as follows:
\begin{theorem}[see~{\cite[Theorem 1.1, p.\ 5]{BRS19}}]\label{thm:SPP-Bogachev-Roeckner-Shaposhnikov-superposition-2019}
  Suppose that $( \mu_{t,n} )_{t \in [0,T]}$ is a ``narrowly continuous'' solution to the finite-dimensional FPKE on $[0,T]$ with initial measure $\nu$ and Assumptions \ref{Assumption:SPP-S0-A-symmetric} -- \ref{Assumption:SPP-S2-integrability-a-b} are fulfilled.
  Then there exists a Borel probability measure $P_n^{\nu}$ on $C([0,T];\mathbb R^n)$ such that the following properties hold:
  \begin{enumerate}[label=\textbf{(m\arabic*)}, leftmargin=40pt]
    \item\label{cond:SP-Rd-MS-BRS19-1} For all Borel sets $B \subseteq \mathbb R^n$ we have $P_n^{\nu} \big[ x_n \in C([0,T];\mathbb R^n) \mid x_n(0) \in B \big] = \nu(B)$.
    \item\label{cond:SP-Rd-MS-BRS19-2} For every function $f \in C^\infty_c (\mathbb R^n)$, the function
          \begin{align*}
            (t,x_n) \longmapsto f(x_n(t)) - f(x_n(0)) - \int_0^t L f(s, x_n(s)) \, \mathrm{d}s
          \end{align*}
          is a martingale with respect to the measure $P_n^{\nu}$ and the natural filtration $\mathcal F_t^{(n)}= \sigma(x_n(s) \mid s\in [0,t])$.  
    \item\label{cond:SP-Rd-MS-BRS19-3} For every function $f \in C^\infty_c (\mathbb R^n)$, the equality
          \begin{align*}
            \int_{\mathbb R^n} f(y) \,  \mu_{t,n}(\mathrm{d}y) = \int_{C([0,T];\mathbb R^n)} f(x_n(t)) \, P_n^{\nu} ( \mathrm{d} x_n)
          \end{align*}
          holds for all $t \in [0,T]$.
  \end{enumerate}
\end{theorem}
Here, the term ``narrowly continuous'' means that we have continuity with respect to the weak topology.

\section{Proofs of Lemma \ref{thm:SP-onH-linear} and Theorem \ref{cor:SP-onH-linear-anySolutionwith}}\label{section:proof-of-thm1}
A more elaborate version with all details can be found in \cite[Section 6.4, p.\ 81ff]{Die20}.
\begin{proof}[Proof of Lemma \ref{thm:SP-onH-linear}]
Let us divide the proof into nine steps.

\smallskip
\noindent
\textbf{Step 1: Starting point}\\
We are given an initial value $x_0 \in \mathbb H$ and coefficients $b$ and $\sigma$ on $[0,T] \times \mathbb H$ which directly allow us to study the martingale problem \eqref{eq:martingale-problem} on $\mathbb H$.
As described in Section \ref{section:framework}, we then also consider the components $b^i$ and $a^{ij}$ that are extended to $[0,T] \times \mathbb R^\infty$ by $0$ for Equation \eqref{eq:FPKE-linear} on $[0,T] \times \mathbb R^\infty$ with initial measure $\delta_{x_0}$ at the same time.
Note that these extensions still satisfy all assumptions that we have imposed in Subsection \ref{subsection:linear-assumptions}.

Now, for every $n \in \mathbb N$, we project the coefficients and initial value/measure down onto $\mathbb H_n$ via the projections $\Pi_n$ to obtain coefficients $\Pi_n b$ and $\Pi_n \sigma$ for the finite-dimensional martingale problem \eqref{eq:MP-Hn} as well as $\Pi_n b$ and $\Pi_n A \, \Pi_n^*$ for the finite-dimensional Equation \eqref{eq:FPKE-linear-Hn}.

Since we have assumed \ref{Assumption:lFPKE-sym-nonneg}--\ref{Assumption:lFPKE-Coefficients}, we can conclude, as in \cite[p.\ 1014]{BDRS15}, existence of solutions $(\mu_{t,n})_{t \in [0,T]}$ to Equation \eqref{eq:FPKE-linear-Hn} for any $n \in \mathbb N$ with the property that the function
\begin{align}\label{eq:lFPKE-continuity-fin-dim-solutions}
  t\mapsto \int_{\mathbb H_n} \zeta(y)\, \mu_{t,n}(\mathrm{d}y)
\end{align}
is continuous on $t\in [0,T]$ for every $\zeta\in C_c^\infty(\mathbb H_n)$ by using \cite[Corollary 3.4, p.\ 415]{BDR08-CPDE}.
In fact, the coefficients $A_n$ and $b_n$ are exactly $\Pi_n b$ and $\Pi_n A \, \Pi_n^*$ in our setting.
To those probability measures $\mu_{t,n}$ on $\mathbb H_n$ we will now apply the superposition principle.

\bigskip
\noindent
\textbf{Step 2: Application of Theorem \ref{thm:SPP-Bogachev-Roeckner-Shaposhnikov-superposition-2019}}\\
Let us fix some $n \in \mathbb N$ for the moment and specify how to exactly apply the finite-dimensional superposition principle on $\mathbb H_n$.
We start with a solution $(\mu_{t,n})_{t \in [0,T]}$ to Equation \eqref{eq:FPKE-linear-Hn} with initial distribution $\delta_{x_0} \circ \Pi_n^{-1}$.
Now, let us check the necessary assumptions for using Theorem \ref{thm:SPP-Bogachev-Roeckner-Shaposhnikov-superposition-2019}.
Assumption \ref{Assumption:SPP-S0-A-symmetric} follows from Assumption \ref{Assumption:lFPKE-sym-nonneg}, Assumption \ref{Assumption:SPP-S1-coefficients-a-b} is fulfilled, since we have $a^{ij}, b^{i} \in L^1_{\mathrm{loc}}(\mu_{t,n} \, \mathrm{d}t)$ for our solution $(\mu_{t,n})_{t \in [0,T]}$ by definition, and Assumption \ref{Assumption:SPP-S2-integrability-a-b} for the projected coefficients $\Pi_n b$ and $\Pi_n A \, \Pi_n^*$ can be derived by in particular using Assumption \ref{Assumption:lFPKE-Coefficients}.

Note that there are two small but apparent differences in the notions of a probability solution in finite-dimensions used in \cite{BRS19} and \cite{BDRS15} (which is in fact constructed in the sense of \cite[p.\ 397f]{BDR08-CPDE}).
First, the finite-dimensional analogue of Equation \eqref{eq:lFPKE-solution-equivalent-no-time-dependence} has to hold for every $t \in [0,T]$ in \cite{BRS19}.
This can be concluded by using Lemma 2.1 in \cite[p.\ 399]{BDR08-CPDE} (including the explanation about the limit on p.\ 400).
Second, continuity with respect to the weak topology of a probability solution is part of its definition in \cite{BRS19}.
For a solution constructed in \cite{BDRS15}, this follows from Equation \eqref{eq:lFPKE-continuity-fin-dim-solutions} by a standard approximation argument enlarging the space of test functions.

Therefore, we can apply Theorem \ref{thm:SPP-Bogachev-Roeckner-Shaposhnikov-superposition-2019} and conclude that there exists a probability measure $P_n$ on $\Omega_n$ such that Conditions \ref{cond:SP-Rd-MS-BRS19-1} and \ref{cond:SP-Rd-MS-BRS19-2} hold.
In addition, $P_n$ also satisfies Condition \ref{cond:SP-Rd-MS-BRS19-3}, i.e.~for the time-marginal laws we have
\begin{align*}
  P_n \circ x_n(t)^{-1} = \mu_{t,n}
\end{align*}
for every $t \in [0,T]$.
This implies in particular, that $P_n$ is a martingale solution to the martingale problem \eqref{eq:MP-Hn} satisfying Conditions \ref{Bedingung-1-martingale-solution} and \ref{Bedingung-2-martingale-solution} of Definition \ref{def:martingale-solution}.

\bigskip
\noindent
\textbf{Step 3: Tightness of $(\bar P_n)_{n \in \mathbb N}$}\\
Collect the family $(P_n)_{n \in \mathbb N}$ of all probability measures obtained by the application of the superposition principle for each $n \in \mathbb N$.
Since they are solutions to \eqref{eq:MP-Hn} satisfying Conditions \ref{Bedingung-1-martingale-solution} and \ref{Bedingung-2-martingale-solution} (and by using e.g.~an adaption of the proof in \cite[Chapter VII, \S 2, p.\ 293ff]{RY99} to obtain the ``weak formulation'' of condition (M2) in \cite{GRZ09}), we are actually in the same situation as in the proof of Theorem 4.6 in \cite[p.\ 1739ff]{GRZ09} (see also the proof of Theorem 2.1 in \cite[p.\ 364ff]{RZZ15}).

There (see p.\ 1739), the authors conclude existence of such probability measures (with no additional property except for being a solution to \eqref{eq:MP-Hn}) from a finite-dimensional result in the appendix, which is based on \cite[Theorem 6.1.7, p.\ 144]{SV79}.
In our case, these solutions are just directly ``created'' by the superposition principle.
Hence, since we have assumed \ref{Assumption:A1-Demicontinuity-SDE}--\ref{Assumption:A3-Growth-SDE}, we can repeat all calculations including the extension of $P_n$ to $\bar P_n$ (see p.\ 1741) and the proof of tightness of the family $(\bar P_n)_{n \in \mathbb N}$ (see p.\ 1742).

\bigskip
\noindent
\textbf{Step 4: Tightness of $(\bar\mu_{t,n})_{n \in \mathbb N}$}\\
After extending the measures $\mu_{t,n}$ on $\mathbb H_n$ to $\bar\mu_{t,n}$ on $\mathbb R^\infty$, we can simply follow \cite[p.\ 1014f]{BDRS15} and prove tightness of the family of probability measures $(\bar\mu_{t,n})_{n \in \mathbb N}$ for every fixed $t\in [0,T]$ the same way as it is done there.

\bigskip
\noindent
\textbf{Step 5: Weak convergence on a joint subsequence}\\
By using a ``diagonal argument'', we can modify the index set of the two tight families, which is necessary for selecting a joint index set on which both subsequences converge to a limit helping us to prove that Equation \eqref{eq:SP-onH-marginal-law-property} holds.
Recall that a subset of a tight set of measures is by definition still a tight set and we lose no additional properties by dropping some indices.
More importantly, we can verify that both proofs remain unchanged after that point by considering a smaller index set.

Hence, we can choose a set of indices for which both tight families have a convergent subsequence with the same indices.
Note that we ensure in the process, that the subsequence of probability solutions converges weakly for any $t \in [0,T]$ (see \cite[p.\ 54ff]{Die20} for more details).
Let us, for simplicity, denote these joint subsequences by $(\bar P_{n_k})_{k \in \mathbb N}$ and $ ( \bar\mu_{t,{n_k}} )_{k \in \mathbb N}$ and their limits by $P$ and $\mu_t$, respectively.

\bigskip
\noindent
\textbf{Step 6: $(\mu_t)_{t \in [0,T]}$ is a solution}\\
For proving that $(\mu_t)_{t \in [0,T]}$ is a probability solution to Equation \eqref{eq:FPKE-linear} on $\mathbb R^\infty$ in the sense of Definition \ref{def:linear-FPKE-notion-solution} we can again follow the proof in \cite{BDRS15} (starting on p.\ 1015).
In particular, Estimate \eqref{eq:lFPKE-estimate-sol-Wk} and Equation \eqref{eq:lFPKE-V-finite-mu_t-1} hold.

\bigskip
\noindent
\textbf{Step 7: $P$ is a solution}\\
We can carry over the calculations from \cite{GRZ09} (see p.\ 1736ff) in order to prove that $P \in \mathcal P (\mathbb S)$ is a solution to the martingale problem \eqref{eq:martingale-problem} with initial measure $\delta_{x_0}$.
In particular, Estimate \eqref{eq:MP-result-estimate-solution} holds.

\bigskip
\noindent
\textbf{Step 8: Time-marginal laws for the limit}\\
Finally, we have to prove that
\begin{align*}
  P \circ x(t)^{-1} = \mu_{t}
\end{align*}
holds for every $t \in [0,T]$.
In fact, let $f \in \mathcal F$, where $\mathcal F \subseteq \mathcal FC_c^\infty (\{e_i\})$ is a measure-separating family on $\mathbb R^\infty$ (always exists, see Lemma \ref{Lemma:count_sep_familiy_measure_sep} in the appendix).
Then $f$ is of the form
\begin{align*}
  f(y) &=  g \big(  y^1, \ldots, y^d \big), \quad y \in \mathbb R^\infty,
\end{align*}
for some $d \in \mathbb N$ and $g \in C_c^\infty (\mathbb R^d)$.

Note that Condition \ref{cond:SP-Rd-MS-BRS19-3} not only holds for functions in $C_c^\infty (\mathbb R^n)$, but also by approximation for functions in $C_b^\infty (\mathbb R^n)$.
Furthermore, for $n \geq d$, a function in $C_c^\infty (\mathbb R^d)$ treated as a function on $\mathbb R^n$ is of class $C_b^\infty (\mathbb R^n)$.

Since we have $\bar{\mu}_{t, n_k} \xrightarrow[k \longrightarrow \infty]{w} \mu_t$ on $\mathbb R^\infty$, we know that
\begin{align*}
  \int_{\mathbb R^\infty} h(y) \, \mu_t(\mathrm{d}y) = \lim\limits_{k \rightarrow \infty}  \int_{\mathbb R^\infty} h(y) \,  \bar{\mu}_{t, n_k} (\mathrm{d}y)
\end{align*}
is fulfilled for every $h \in C_b(\mathbb R^\infty)$, i.e.~in particular for $f$.
In addition, we have that $\bar{P}_{n_k} \xrightarrow[k \longrightarrow \infty]{w} P$ on $\Omega$, which means that
\begin{align*}
  \int_{ \Omega} h(\omega) \,  P (\mathrm{d}\omega) = \lim\limits_{k \rightarrow \infty} \int_{ \Omega} h(\omega) \, \bar{P}_{n_k} (\mathrm{d}\omega) 
\end{align*}
 holds for every $h \in C_b(\Omega)$.
Consequently, it is also true for the mapping given by $\omega \longmapsto  g \big(  \omega(t)^1, \ldots, \omega(t)^d \big)$ for every $t \in [0,T]$.

Then we obtain
\begin{align*}
  \int_{\mathbb R^\infty} &f(y) \, \mu_t(\mathrm{d}y) 
     = \lim\limits_{k \rightarrow \infty}  \int_{\mathbb R^\infty} g \big(  y^1, \ldots, y^d \big) \,  \bar{\mu}_{t, n_k} (\mathrm{d}y)
  \\ &= \lim\limits_{k \rightarrow \infty}  \int_{\mathbb H_{n_k}} g \big(  y^1, \ldots, y^d \big) \,  \mu_{t, n_k} (\mathrm{d}y)
  \\ &= \lim\limits_{k \rightarrow \infty}  \int_{ \Omega_{n_k}} g \big(  \omega(t)^1, \ldots, \omega(t)^d \big) \,  P_{n_k} (\mathrm{d}\omega)
  \\ &= \lim\limits_{k \rightarrow \infty}  \int_{ \Omega} g \big(  \omega(t)^1, \ldots, \omega(t)^d \big) \,  \bar{P}_{n_k} (\mathrm{d}\omega) 
  \\ &= \int_{ \Omega} g \big(  \omega(t)^1, \ldots, \omega(t)^d \big) \,  P (\mathrm{d}\omega)
  \\ & = \int_{ \Omega} f(\omega(t)) \,  P (\mathrm{d}\omega) 
      = \int_{ \mathbb S} f(x(t)) \,  P (\mathrm{d}x)
\end{align*}
for any $t \in [0,T]$.
Since $f \in \mathcal F$ separates measures on $\mathbb R^\infty$ (and all of its subsets), the assertion follows.

\bigskip
\noindent
\textbf{Step 9: $\mu_t$ are probability measures on $\mathbb H$}\\
From Estimate \eqref{eq:MP-result-estimate-solution} we can follow that 
\begin{align*}
  \mathbb E^P \Big[\sup\limits_{t \in [0,T]} \|x(t)\|^{2q}_{\mathbb H} \Big] < \infty
\end{align*}
holds for every $q \geq 1$, where we made use of the lower semi-continuity of the norm $\| \cdot \|_{\mathbb H}$ as an extended function on $\mathbb X^*$ and, therefore, of the supremum.
Consequently, $P \circ x(t)^{-1}$ is a probability measure on $\mathbb H$ for every $t \in [0,T]$, hence by Step 8 so is $\mu_t$.
\end{proof}

\begin{proof}[Proof of Theorem \ref{cor:SP-onH-linear-anySolutionwith}]
We can just repeat the proof of Lemma \ref{thm:SP-onH-linear}, because this time we are simply given an explicit family $ \{ (\mu_{t,{n_k}})_{t \in [0,T]} \mid k \in \mathbb N\}$ of solutions to the finite-dimensional Equations \eqref{eq:FPKE-linear-Hn} on $\mathbb H_{n_k}$ for which we already know that, for $k\rightarrow \infty$, $( \bar\mu_{t,{n_k}} )_{t \in [0,T]}$ converges weakly to the given solution $(\mu_t)_{t \in [0,T]}$ of Equation \eqref{eq:FPKE-linear}.

In particular, all steps necessary to derive finite-dimensional probability solutions are redundant, because we have directly assumed all desired properties for $(\mu_t)_{t \in [0,T]}$ and $(\bar \mu_{t,{n_k}})_{t \in [0,T]}$, $k \in \mathbb N$. 
Furthermore, we can apply the finite-dimensional superposition principle, because we have ensured Condition \ref{Assumption:SPP-S2-integrability-a-b} directly by assumption. 
\end{proof}

\section{Application to $d$-dimensional stochastic Navier--Stokes equations}\label{section:navier-stokes}
The articles \cite{BDRS15}, \cite{BKRS15}, \cite{GRZ09} and \cite{RZZ15} contain extensive calculations on $d$-dimensional SNSEs that remain valid in our case and to which we refer for all details.

We will instead focus on the connection between the part on FPKEs in \cite[Example 3.5, p.\ 1018f]{BDRS15}, which can partly also be found in \cite[Example 10.1.6, p.\ 411f and Example 10.4.3, p.\ 425f]{BKRS15}, and the part on martingale problems in \cite[Chapter 6, p.\ 1749ff]{GRZ09} and \cite[Section 5.1, p.\ 377f]{RZZ15} in order to fit everything into our combined framework.

Let us begin with recalling the framework of \cite{BDRS15}.
In particular, we will state the stochastic differential equation under consideration.
We note, that it should be regarded as a heuristic expression with no need for further interpretation because we are rather interested in the coefficients of the equation that we need for the martingale problem as well as in the specific form of the corresponding Kolmogorov operator $L$ that can be derived from them (recall that Lemma \ref{thm:SP-onH-linear} does not explicitly involve the SDE).

\paragraph{Stochastic Navier--Stokes equations \emph{(see \cite{BDRS15})}.}
Let $d \in \mathbb N$, $T>0$ and let $D \subseteq \mathbb{R}^d$ be a bounded domain with smooth boundary.
The stochastic Navier--Stokes equation under consideration is formally written as
\begin{align}\label{eq:Navier-Stokes}\tag{SNS}
\begin{split}
  \mathrm{d} u (t,z)=  \Pi_{\mathbb H} \Big( \Delta_z &u(t,z) - \sum_{j=1}^d u^j(t,z) \partial_{z_j} u(t,z) \Big) \, \mathrm{d}t  + \sqrt{2} \, \mathrm{d} W(t,z),
\end{split}
\end{align}
where we have chosen the ``force'' $F$ for simplicity to be $0$, and are hence in the case of the ``classical'' stochastic Navier--Stokes equation.

\paragraph{Setting \emph{(see \cite{BDRS15}, \cite{GRZ09})}.}
For the space $L^2(D; \mathbb R^d)$, we denote the $L^2$-inner product by $\langle \cdot, \cdot \rangle_{L^2}$ and the corresponding $L^2$-norm by $\| \cdot \|_{L^2}$.
Let 
\begin{align*}
  \mathcal V:= C^\infty_{c, \text{div}} (D)^d  := \big\{ u=(u^1,\ldots,u^d) \, \big| \, u^j\in C_c^{\infty}(D), \Div u=0 \big\}
\end{align*}
be the space of all smooth $d$-dimensional vector fields on $D$ with compact support in $D$ and divergence free.
By $H_0^{2,1}(D)$ we denote the closure of $C^\infty_{c} (D)$ in the Sobolev space $H^{2,1}(D)$.
As in \cite{BDRS15}, we define the space 
\begin{align*}
  V_2 := H_{0, \text{div}}^{2,1}(D)^d := \big\{ u=(u^1,\ldots,u^d) \, \big| \, u^j\in H_0^{2,1}(D), \Div u=0 \big\}
\end{align*}
of $\mathbb{R}^d$-valued mappings.
Then $V_2$ is equipped with its natural Hilbert norm $\| \cdot \|_{V_2}$ given by
\begin{align*}
\|u\|_{V_2}^2:=\sum_{j=1}^d \|\nabla_z u^j\|_{L^2}^2.
\end{align*}
Note that $V_2$ is identical to the closure of $\mathcal V$ in $H_0^{2,1}(D)^d$ (see e.g.~\cite[Theorem 1.6, p.\ 18 and Remark 2.1(ii), p.\ 23]{Tem77}).
Let $\mathbb H$ be the closure of $\mathcal V$ in $H^{2,0}_{0}(D)^d$ as in \cite{GRZ09}.
Then $\mathbb H$ is clearly a separable Hilbert space and is identical to the closure of $V_2$ in $L^2(D;\mathbb{R}^d)$ as it is constructed in \cite{BDRS15}.
As in \cite{GRZ09}, let $\mathbb X$ be the closure of $\mathcal V$ in the Sobolev space $H^{2,2+d}_{0}(D)^d$ and let $\mathbb X^*$ be its dual.

It is well-known that there exist eigenfunctions $\eta_i\in \mathbb X$ of the Laplace operator $\Delta_z$ with eigenvalues $-\lambda_i^2$, i.e.~$\Delta_z \eta_i = -\lambda_i^2 \eta_i$, $\lambda>0$, such that $\{\eta_i \mid i \in \mathbb N \}$ is an orthonormal basis in $\mathbb H$.
Define $\mathbb H_n := \Span\{\eta_1, \ldots \eta_n\}$.

Altogether, we consider the embedding
\begin{align*}
  \mathbb X \subseteq V_2 \subseteq \mathbb H \subseteq V_2^* \subseteq \mathbb X^*.
\end{align*}
Furthermore, $\Pi_{\mathbb H}$ denotes the orthogonal projection onto $\mathbb H$ in $L^2(D;\mathbb{R}^d)$.
Let $W$ be a Wiener process of the form 
\begin{align*}
  W(t,z)=\sum_{i=1}^\infty \sqrt{\alpha_i} w_i(t) \eta_i(z),
\end{align*}
where $\alpha_i \geq 0$, $\sum_{i=1}^\infty \alpha_i<\infty$ and $w_i$, $i \in \mathbb N$, are independent real Wiener processes.

\paragraph{Coefficients \emph{(see \cite{BDRS15}, \cite{GRZ09}, \cite{RZZ15})}.}
For $t\in [0,T]$ and $v \in \mathcal V$ set
\begin{align*}
  b(t,v):= \Pi_{\mathbb H} \Delta_z v - \Pi_{\mathbb H} \sum_{j=1}^d v^j \partial_{z_j} v.
\end{align*}
Note that the following lemma holds, which can be proved by similar calculations as in \cite[Lemma 6.1, p.\ 1750]{GRZ09}.
\begin{lemma}[{see \cite[p.\ 378]{RZZ15}}]\label{lemma:applications}
For any $v_1,\tilde v_1, v_2, \tilde v_2 \in\mathcal V$ we have
\begin{align*}
    \big\| \Pi_{\mathbb H} \Delta_z v_1 - \Pi_{\mathbb H} \Delta_z v_2 \big\|_{\mathbb X^*} &\leq C \|v_1-v_2\|_{\mathbb H},
\\  \Big\| \Pi_{\mathbb H} \sum_{j=1}^d {\tilde v_1}^j \partial_{z_j} v_1 - \Pi_{\mathbb H} \sum_{j=1}^d {\tilde v_2}^j \partial_{z_j} v_2 \Big\|_{\mathbb X^*} &\leq C \big(\|{\tilde v_1}\|_{\mathbb H} \|v_1-v_2\|_{\mathbb H} +\|v_2\|_{\mathbb H} \|\tilde v_1-\tilde v_2\|_{\mathbb H} \big) .
\end{align*}
\end{lemma}

Since $\langle \Pi_{\mathbb H} w,\eta_i\rangle_{L^2}=\langle w,\eta_i\rangle_{L^2}$ holds for any $w\in L^2(D;\mathbb{R}^d)$,
we can, however, directly consider the components $b^i$ of our coefficient $b$ given by
\begin{align*}
b^i(t,v) 
   &= \langle \Pi_{\mathbb H} ( \Delta_z v), \eta_i\rangle_{L^2} -\sum_{j=1}^d \langle \Pi_{\mathbb H} (v^j\partial_{z_j} v), \eta_i\rangle_{L^2} =\langle v,\Delta_z \eta_i\rangle_{L^2} -\sum_{j=1}^d \langle \partial_{z_j}v, v^j\eta_i\rangle_{L^2}.
\end{align*}
It follows from the last step that those mappings are defined for every $v \in V_2$.
Since $v$ is divergence free, we can further rewrite $b^i$ by using integration by parts as
\begin{align*}
  b^i(t,v) =\langle v,\Delta_z \eta_i\rangle_{L^2} +\sum_{j=1}^d \langle v, v^j \partial_{z_j} \eta_i\rangle_{L^2},
\end{align*}
which is actually defined for all $v \in \mathbb H$.
Furthermore, for the coefficient $\sigma$, we set
\begin{align*}
\sigma^{ij} :\equiv \begin{cases} \sqrt{2 \alpha_i}, & i=j, \\0 , & i \neq j \end{cases}
\end{align*}
on $\mathbb H$ and obtain for $A$ the components
\begin{align*}
  a^{ij} :\equiv \begin{cases} \alpha_i, & i=j, \\0 , & i \neq j. \end{cases}
\end{align*}
Hence, the treatment of $\sigma$ and $A$ as constant coefficients is straightforward and for $b$ we know, that the components $b^i$ are defined on the whole space $\mathbb H$ and their representation on $V_2$ is directly derived from the stochastic Navier--Stokes equation \eqref{eq:Navier-Stokes}.
Recall that both $\mathbb H_n$ and $\mathcal E = \Span \{\eta_1, \eta_2, \ldots \}$, which is used in the definition of a martingale solution, are subspaces of $V_2$, simplifying many calculations even more.

\paragraph{Kolmogorov operator $L$.}
Now, we can consider the Kolmogorov operator $L$ given by
\begin{align*}
    L\varphi(t,u)= \sum_{i=1}^\infty \alpha_i \, \partial_{\eta_i} \partial_{\eta_i} \varphi (u) +\sum_{i=1}^\infty b^i(t,u) \,\partial_{\eta_i}\varphi (u),
\end{align*}
for any $\varphi \in \mathcal FC_c^\infty (\{ \eta_i \})$.

\paragraph{Assumptions.}
Note that we had to modify two small parts of Assumptions \ref{Assumption:lFPKE-Theta} and \ref{Assumption:lFPKE-V} in Subsection \ref{subsection:linear-assumptions} compared to those in \cite{BDRS15} and \cite{BKRS15} for being able to prove Lemma \ref{thm:SP-onH-linear}.
Let us explain that these modifications have to influence on the calculations.
First of all, since we choose the Lyapunov function $V \colon \mathbb R^\infty \longrightarrow [1,\infty]$ as
  \begin{align*}
    V(u)=\begin{cases} \|u\|_{L^2}^2+1, &u \in \mathbb H, \\ \infty, &\text{else}, \end{cases}
  \end{align*}
the restriction of $V$ to $\mathbb H_n$ is non-degenerate, as required, because it is a convex function.
Second, the function $\Theta \colon \mathbb R^\infty \longrightarrow [0,\infty]$, which we choose as
  \begin{align*}
    \Theta(u) = \begin{cases} C_1\|u\|_{V_2}^2, &u \in V_2, \\ \infty, &\text{else}, \end{cases}
  \end{align*}
is bounded on bounded sets on each space $\mathbb H_n$, which is the other modification.

Hence, Assumptions \ref{Assumption:lFPKE-sym-nonneg}--\ref{Assumption:lFPKE-Coefficients} are verified as in \cite{BDRS15} and \cite{BKRS15}.
The validation of Assumptions \ref{Assumption:N1-MP-SDE}, \ref{Assumption:A1-Demicontinuity-SDE}--\ref{Assumption:A3-Growth-SDE} remains unchanged from \cite{GRZ09} and \cite{RZZ15}.

\paragraph{Conclusion.}
 Lemma \ref{thm:SP-onH-linear} and  Theorem \ref{cor:SP-onH-linear-anySolutionwith} apply.

\section{Nonlinear framework}\label{section:framework-nonlinear}
As a generalization we will now consider nonlinear Fokker--Planck--Kolmogorov equations in infinite dimensions.

Our setting remains, for the most part, identical to Section \ref{section:framework}.
Let us focus on what exactly changes in the nonlinear case.
For some fixed $T>0$, let the mappings
\begin{align*}
  b &\colon [0,T] \times \mathbb H \times  \mathcal{P}(\mathbb H) \longrightarrow \mathbb X^*,\\
  \sigma &\colon [0,T] \times \mathbb H \times  \mathcal{P}(\mathbb H) \longrightarrow L_2(\mathbb U, \mathbb H)
\end{align*}
be Borel-measurable, where $\mathcal{P}(\mathbb H)$ denotes the space of all Borel probability measures on $\mathbb H$.
By extending those mappings by 0 as before, we again obtain mappings
\begin{align*}
  a^{ij} &\colon [0,T] \times \mathbb R^\infty \times  \mathcal{P}(\mathbb R^\infty) \longrightarrow \mathbb R,\\
  b^{i} &\colon [0,T] \times \mathbb R^\infty \times  \mathcal{P}(\mathbb R^\infty) \longrightarrow \mathbb R
\end{align*}
as coefficients. 
To be more precise, we now set
\begin{align*}
  b^{i}(t,y, \varrho) := \begin{cases} {}_{\mathbb X^*}^{}\langle b(t,y, \varrho), e_i \rangle_{\mathbb X}, &(t,y,\varrho) \in [0,T] \times \mathbb H \times  \mathcal{P}(\mathbb H), \\ 0, & \text{else} \end{cases}
\end{align*}
and
\begin{align*}
  a^{ij}(t,y, \varrho) := \begin{cases} \frac12 \langle \sigma(t,y,\varrho) \sigma(t,y,\varrho)^* e_i,  e_j \rangle_{\mathbb H} , &(t,y,\varrho) \in [0,T] \times \mathbb H \times  \mathcal{P}(\mathbb H), \\ 0, & \text{else}. \end{cases}
\end{align*}
Again, let $A(t,y,\varrho):=\big( a^{ij}(t,y, \varrho) \big)_{1 \leq i,j < \infty}$ be our diffusion matrix.

\subsection{Kolmogorov operator and nonlinear FPKE}

Let $L_{\mu_t}$ be the Kolmogorov operator, acting on functions $\varphi \in \mathcal FC^{2}(\{e_i\})$, which is given by
\begin{align*}
  L_{\mu_t} \varphi (t,y) = \sum\limits_{i,j= 1}^{\infty} a^{ij}(t, y, \mu_t) \partial_{e_i} \partial_{e_j} \varphi(y) + \sum\limits_{i= 1}^{\infty} b^{i}(t, y, \mu_t) \partial_{e_i} \varphi(y)
\end{align*}
for $(t,y) \in [0,T] \times \mathbb R^\infty$.

Consider the following shorthand notation for an infinite-dimensional nonlinear Fokker--Planck--Kolmogorov equation associated to the Kolmogorov operator $L_{\mu_t}$, which is a Cauchy problem of the form 
\begin{align}\label{FPKE-nonlinear}\tag{NFPKE}
\begin{split}
  \partial_t \mu_t &= L_{\mu_t}^{*} \, \mu_t,
\\ \mu_{0} &= \delta_{{x_0}},
\end{split}
\end{align}
where $L_{\mu_t}^*$ is the formal adjoint of the operator $L_{\mu_t}$ and the solution $(\mu_t)_{t \in [0,T]}$ is a path of Borel probability measures on $\mathbb R^\infty$.

\begin{remark}
We will, as in Section \ref{section:framework}, reduce our calculations to Dirac measures $\delta_{{x_0}}$ as initial measures.
This is again only for simplicity and without implications for general initial measures. 
However, as mentioned in Section \ref{section:framework}, it is not possible anymore to generalize our results to any initial measures by simply integrating over the $\delta_{{x_0}}$ in the nonlinear setting since the FPKE's linear dependence on the initial measure is lost.
\end{remark}

In addition, let us introduce the finite-dimensional nonlinear FPKE on $\mathbb H_n$, that can be written in short-hand notation as
\begin{align}\label{eq:FPKE-nonlinear-Hn}\tag{$\mathrm{NFPKE}_n$}
\begin{split}
\partial_t \mu_{t,n} &= L^*_{n, \mu_{t,n}} \, \mu_{t,n},
\\ \mu_{0,n} &= \delta_{x_0} \circ \Pi_n^{-1}.
\end{split}
\end{align}
Here, the operator $L_{n, \mu_{t,n}}$, acting on functions $\varphi \in C^{2} (\mathbb H_n)$, is given by
\begin{align*}
  L_{n, \mu_{t,n}} \varphi (t,y) = \sum\limits_{i,j= 1}^{n} a^{ij}(t,y, \mu_{t,n}) \partial_{e_i} \partial_{e_j} \varphi(y) + \sum\limits_{i= 1}^{n} b^{i}(t,y, \mu_{t,n}) \partial_{e_i} \varphi(y),
\end{align*}
for $(t,y) \in [0,T] \times \mathbb H_n$.

\subsection{Notion of solution}
The notion of a probability solution in the nonlinear case is analogue to Definition \ref{def:linear-FPKE-notion-solution}.
We only have to consider coefficients explicitly depending on a measure.

\begin{definition}(probability solution, nonlinear)\label{def:prob-sol-FPKE-nonlinear}
A path $(\mu_t)_{t \in [0,T]}$ of Borel probability measures on $\mathbb R^\infty$
is called probability solution to Equation \eqref{FPKE-nonlinear} if 
\begin{enumerate}[label=(\roman*), leftmargin=30pt]
 \item\label{condition:NFPKE-probability-solution-i-L1} The functions $a^{ij}$, $b^{i}$ are integrable with respect to the measure $\mu_t \, \mathrm{d}t$, i.e.
 \begin{align*}
    (t,x) \longmapsto a^{ij}(t, x, \mu_{t}), (t,x) \longmapsto b^{i}(t, x, \mu_{t}) \in L^1([0,T]  \times \mathbb R^\infty, \mu_t \, \mathrm{d}t).
 \end{align*}
 \item\label{condition:probability-solution-ii-integral} For every function $\varphi\in \mathcal FC_c^\infty (  \{e_i\} )$ we have
\begin{align*}
  \int_{\mathbb R^\infty} \varphi(y) \, \mu_t(\mathrm{d}y)  = \int_{\mathbb R^\infty} \varphi(y) \, \delta_{x_0} (\mathrm{d}y) + \int_0^t \int_{\mathbb R^\infty} L_\mu \varphi(s,y) \, \mu_s(\mathrm{d}y) \, \mathrm{d}s
\end{align*}
for every $t \in [0,T]$.
\end{enumerate}
\end{definition}

The notion of a martingale solution remains completely unchanged because, by starting the statement logic of the superposition principle with a given solution $(\mu_t)_{t\in[0,T]}$ to \eqref{FPKE-nonlinear}, we can fix it and are, therefore, back in a linear situation as it has been introduced in Section \ref{section:framework}.
To be more precise, we will simply consider martingale problems with ``special'' functions in the place of $b$ and $\sigma$ that explicitly depend on a fixed measure, i.e.~the mappings $(t,x) \longmapsto b(t, x, \mu_{t})$ and $(t,x) \longmapsto \sigma(t, x, \mu_{t})$, in the following.
Hence, the martingale problem we study can be stated as:
\begin{align}\label{eq:McKV-MP}\tag{NMP}
\begin{tabular}{c}
 \text{Existence of a martingale solution $P \in \mathcal P(\mathbb S)$ in the sense of}\\
 \text{Definition \ref{def:martingale-solution} for the coefficients $(t,x) \longmapsto b(t, x, \mu_{t})$ and}\\
 \text{$(t,x) \longmapsto \sigma(t, x, \mu_{t})$ and with initial value $x_0 \in \mathbb H$}.
\end{tabular}
\end{align}

\begin{remark}
As before, by following \cite[Theorem 2.2, p.\ 364]{RZZ15} and using \cite[Theorem 2, p.\ 1007]{Ond05}, we can establish the connection to a weak solution of the corresponding McKean--Vlasov SDE.
\end{remark}

\subsection{Nonlinear assumptions}\label{section:nonlinear-Assumptions}
The following assumptions on the coefficients $b$ and $\sigma$ are directly adapted from those in Subsection \ref{subsection:linear-assumptions} by making the estimates uniform in the newly added dependence on measures.

\begin{enumerate}[label=\textbf{(NN)}, leftmargin=40pt]
  \item \label{Assumption:nonlinear-N1} Assume there exists a function $\mathcal N \in \mathfrak U^p$ for some $p \geq 2$ such that for every $n \in \mathbb N$ there exists a constant $C_n \geq 0$ with
      \begin{align*}
      \mathcal N(y) \leq C_n \| y \|^p_{\mathbb H_n},
    \end{align*}
    for any $y \in \mathbb H_n$.
\end{enumerate}

\begin{enumerate}[label=\textbf{(NA\arabic*)}, leftmargin=40pt]
 \item (Demicontinuity)\label{Assumption:nonlinear-A1} For any $v \in \mathbb X$, $t \in [0,T]$, $\varrho \in \mathcal{P}(\mathbb H)$ and every sequence $(y_k)_{k \in \mathbb N}$ with $y_k \xrightarrow[k \rightarrow \infty]{} y$ in $\mathbb H$, we have 
 \begin{align*}
 \lim\limits_{k \rightarrow \infty} {}_{\mathbb X^*}\langle b(t, y_k, \varrho),v \rangle_{\mathbb X} = {}_{\mathbb X^*}\langle b(t, y, \varrho),v \rangle_{\mathbb X}
 \end{align*}
  and
 \begin{align*}
 \lim\limits_{k \rightarrow \infty} \big\| \sigma^* (t, y_k, \varrho ) (v) -  \sigma^* (t, y, \varrho ) (v) \big\|_{\mathbb U} = 0.
 \end{align*}

\item (Coercivity)\label{Assumption:nonlinear-A2} There exists a constant $\lambda_1 \geq 0$ such that for all $y \in \mathbb X$, $t \in [0,T]$ and $\varrho \in \mathcal{P}(\mathbb H)$
 \begin{align*}
   {}_{\mathbb X^*}\langle b(t, y, \varrho),y \rangle_{\mathbb X} \leq - \mathcal N(y) + \lambda_1 (1+ \|y\|^2_{\mathbb H})
  \end{align*}
holds.

\item (Growth)\label{Assumption:nonlinear-A3} There exist constants $\lambda_2, \lambda_3, \lambda_4 > 0$ and constants $\gamma' \geq \gamma > 1$ such that for all $y \in \mathbb H$, $t \in [0,T]$ and $\varrho \in \mathcal{P}(\mathbb H)$ we have
  \begin{align*}
    \|b(t,y, \varrho)\|^\gamma_{\mathbb X^*} &\leq \lambda_2 \, \mathcal N(y) + \lambda_3 (1+ \|y\|^{\gamma'}_{\mathbb H})
  \end{align*}
  and
  \begin{align*}
    \|\sigma(t,y, \varrho)\|^2_{L_2 (\mathbb U ; \mathbb H)} &\leq \lambda_4 (1+ \|y\|^2_{\mathbb H}).
  \end{align*}
\end{enumerate}

\begin{enumerate}[label=\textbf{(NH\arabic*)}, leftmargin=40pt]
\item \label{Assumption:nonlinear-H1} For all $n \in \mathbb N$, $\varrho \in \mathcal{P}(\mathbb H)$, the matrices $(a^{ij}(\cdot,\cdot,\varrho))_{1 \leq i,j \leq n}$ are symmetric and nonnegative definite.
\end{enumerate}

\section{Nonlinear result}\label{section:nonlinear-results}

Let us state the main result for nonlinear FPKEs, which is a direct adaption of Theorem \ref{cor:SP-onH-linear-anySolutionwith}.
To be more precise, it is a superposition principle for a given probability solution $(\mu_t)_{t \in [0,T]}$ on $\mathbb H$ to a nonlinear FPKE yielding existence of a martingale solution whose time-marginals are equal to $\mu_t$.

\begin{theorem}\label{thm:SP-nonlinear-onH}
Let the assumptions from Subsection \ref{section:nonlinear-Assumptions} be fulfilled.
Assume there exists a probability solution $(\mu_t)_{t \in [0,T]}$ on $\mathbb H$ to the nonlinear Equation \eqref{FPKE-nonlinear} in the sense of Definition \ref{def:prob-sol-FPKE-nonlinear} and a subsequence $\{ ( \mu_{t,n_k} )_{t \in [0,T]} \mid k \in \mathbb N\}$ of paths of Borel probability measures on $\mathbb H_{n_k}$ with the following properties:
 \begin{itemize}
    \item[$\bullet$] The paths $(\mu_{t,{n_k}})_{t \in [0,T]}$, $k \in \mathbb N$, are solutions to the finite-dimensional nonlinear Equations \eqref{eq:FPKE-nonlinear-Hn} on $\mathbb H_{n_k}$ with the property that the mapping
        \begin{align*}
          t\mapsto \int_{\mathbb H_{n_k}} \zeta(y)\, \mu_{t,n_k}(\mathrm{d}y)
        \end{align*}
        is continuous on $[0,T]$ for every $\zeta\in C_c^\infty(\mathbb H_{n_k})$.
    \item[$\bullet$] For the family $( \bar\mu_{t,{n_k}} )_{k \in \mathbb N}$ of extended measures to $\mathbb H$, we have $\bar \mu_{t, n_k} \xrightarrow[k \rightarrow \infty]{w} \mu_t$ for every $t \in [0,T]$.
    \item[$\bullet$] The integrability condition
    \begin{align*}
      \int_0^T \int_{\mathbb H_{n_k}} \frac{\| \Pi_{n_k} A(t,y, \mu_{t, n_k}) \Pi_{n_k}^* \| + | \langle \Pi_{n_k} b(t,y, \mu_{t, n_k}), y \rangle_{\mathbb H_{n_k}} |}{(1+\|y\|_{\mathbb H_{n_k}})^2} \, \mu_{t,_{n_k}}(\mathrm{d}y) \, \mathrm{d}t < \infty
    \end{align*}
    holds for every $k \in \mathbb N$.
\end{itemize}
Then there exists a martingale solution $P \in \mathcal P(\mathbb S)$ to the martingale problem \eqref{eq:McKV-MP} in the sense of Definition \ref{def:martingale-solution}, for which 
\begin{align*}
  P \circ x(t)^{-1} = \mu_{t}
\end{align*}
holds for every $t \in [0,T]$.
\end{theorem}

As mentioned before, we follow the ideas presented in \cite{BR18a} and \cite[Section 2]{BR18b} on dealing with nonlinearity for the proof.

\begin{proof}[Proof of Theorem \ref{thm:SP-nonlinear-onH}]
Given solutions $(\mu_t)_{t\in [0,T]}$ to the nonlinear Equation \eqref{FPKE-nonlinear} and $(\mu_{t,{n_k}})_{t \in [0,T]}$, $k \in \mathbb N$, to the finite-dimensional nonlinear Equations \eqref{eq:FPKE-nonlinear-Hn}, we ``freeze'' all of these measures and consider linear FPKEs of the form
\begin{align*}
    \partial_t \varrho_t = L_{\mu_t}^{*} \, \varrho_t 
\end{align*}
and, for $k \in \mathbb N$,
\begin{align*}
    \partial_t \varrho_t = L_{n_k,\mu_{t,n_k}}^{*} \, \varrho_t. 
\end{align*}
But then, $(\mu_t)_{t \in [0,T]}$ and $(\mu_{t,{n_k}})_{t \in [0,T]}$, for $k \in \mathbb N$, are again particular solutions of these linear FPKEs.
Since our assumptions from Section \ref{section:nonlinear-Assumptions} are uniform in the dependence on the measure, all assumptions from Theorem \ref{cor:SP-onH-linear-anySolutionwith} hold for our new coefficients $(t,x) \longmapsto b(t, x, \mu_{t})$ and $(t,x) \longmapsto \sigma(t, x, \mu_{t})$ in the linear case.
It follows that we can use Theorem \ref{cor:SP-onH-linear-anySolutionwith} to obtain the desired martingale solution $P$ in the sense of Definition \ref{def:martingale-solution}.
\end{proof}

\appendix
\section{Measure-separating families of finitely based functions}
Let us follow up on the lemma on countable measure-separating families of finitely based functions, which can be proved by using similar techniques as in \cite[p.\ 119]{MR92}.

\begin{lemma}\label{Lemma:count_sep_familiy_measure_sep}
There exits a countable family $\mathcal F$ of functions in $\mathcal FC^\infty_c (\{e_i\})$, which separates measures on $\mathcal B( \mathbb R^\infty )$ (i.e.~for any two measures $\mu_1, \mu_2 \in \mathcal B( \mathbb R^\infty )$ with $\mu_1 \neq \mu_2$ there exists $f \in \mathcal F$ such that $\int_{\mathbb R^\infty} f \, \mathrm{d} \mu_1 \neq \int_{\mathbb R^\infty} f \, \mathrm{d} \mu_2$).
\end{lemma}

\begin{proof}
Let us begin by first showing the following two simplified claims for point-separating families.
\smallskip
\\
\noindent
\textbf{Claim 1:} There exists a family $\tilde{\mathcal F} \subseteq \mathcal FC^\infty_c(\{e_i\})$ that separates points in $\mathbb R^\infty$.
\smallskip
\\
\noindent
\textbf{Proof of Claim 1:}
Let $y_1, y_2 \in \mathbb R^\infty$ with $y_1 \neq y_2$.
Consequently, there exists some $d \in \mathbb N$ such that their $d$-th component differs, i.e.~$y_1^d \neq y_2^d$.
We consider
\begin{align*}
      \Pi_d^\infty y_1 = ( y_1^1, \ldots , y_1^d) \in \mathbb R^d \quad \text { and } \quad  \Pi_d^\infty y_2 = ( y_2^1, \ldots , y_2^d) \in \mathbb R^d,
\end{align*}
where $\Pi_d^\infty$ is the projection onto $\mathbb R^d$, and obtain $\Pi_d^\infty y_1 \neq \Pi_d^\infty y_2$.
By using the fact that points in $\mathbb R^d$ can be separated by functions of class $C^\infty_c(\mathbb R^d)$, there exists some $f \in C^\infty_c(\mathbb R^d)$ such that $f(\Pi_d^\infty y_1) \neq f(\Pi_d^\infty y_2)$.
Hence, we can just consider the finitely based function $\xi \colon \mathbb R^\infty \longrightarrow \mathbb R$ given by $\xi(y):= f (\Pi_d^\infty (y) )$,  $y \in \mathbb R^\infty$.
The family $\tilde{\mathcal F}$ is then chosen to be the collection of all such functions $\xi$.
\qed
~
\\
\smallskip
\noindent
\textbf{Claim 2:} There exits a countable family $\mathcal F \subseteq\mathcal FC^\infty_c (\{e_i\})$ that separates points in $\mathbb R^\infty$.
\smallskip
\\
\noindent
\textbf{Proof of Claim 2:}
By using Claim 1 from above, it remains to show that there exists a subset $\mathcal F$ of the family $\tilde{\mathcal F}$ that is in fact countable.

For any function $f \in \mathcal FC^\infty_c(\{e_i\})$, let $(f,f) \colon \mathbb R^\infty \times \, \mathbb R^\infty \longrightarrow \mathbb R \times \mathbb R$ be the function given by $(f, f)(y_1, y_2):= (f(y_1), f(y_2))$.
Set $D_{\mathbb R^\infty} := \{ (y_1,y_2) \in \mathbb R^\infty \times \, \mathbb R^\infty \mid y_1=y_2\}$ to be the diagonal of $\mathbb R^\infty  \times \, \mathbb R^\infty$ and analogously let $D_{\mathbb R}$ be the diagonal of $\mathbb R \times \mathbb R$.
Then, since the functions $f \in \tilde{\mathcal F}$ separate points in $\mathbb R^\infty$, the equation
\begin{align}\label{eq:lemma-cover-diagonal}
  \big(\mathbb R^\infty \times \, \mathbb R^\infty \big) \setminus D_{\mathbb R^\infty} = \bigcup\limits_{f \in \tilde{\mathcal F}} (f,f)^{-1} (\mathbb R \times \mathbb R \setminus D_{\mathbb R})
\end{align}
holds.

Furthermore, note that $\mathbb R^\infty \times \, \mathbb R^\infty$ is a Polish space and, hence, a strongly Lindelöf space.
Since $\big( \mathbb R^\infty \times \, \mathbb R^\infty\big) \setminus D_{\mathbb R^\infty}$ is an open subset of $\mathbb R^\infty \times \, \mathbb R^\infty$, the open cover on the right hand side of Equation \eqref{eq:lemma-cover-diagonal} has a countable subcover, i.e.~there exists a countable family $\mathcal F \subseteq \mathcal FC^\infty_c (\{e_i\})$ such that
\begin{align*}
  \big( \mathbb R^\infty \times \, \mathbb R^\infty\big) \setminus D_{\mathbb R^\infty} = \bigcup\limits_{f \in \mathcal F} (f,f)^{-1} (\mathbb R \times \mathbb R \setminus D_{\mathbb R}).
\end{align*}
The assertion follows.
\qed

Finally, let us prove the statement of the lemma. 
Without loss of generality we can assume that $\mathcal F$ is a multiplicative system.
By a monotone class argument it then remains to prove that $\mathcal F$ generates $\mathcal B (\mathbb R^\infty)$.
Since the functions in $\mathcal F$ are continuous, we immediately have $\sigma (\mathcal F) \subseteq \mathcal B (\mathbb R^\infty)$.
 
In addition, we consider the measurable function $id$ mapping from the Polish space $(\mathbb R^\infty, \mathcal B (\mathbb R^\infty))$ to the space $(\mathbb R^\infty, \sigma (\mathcal F))$ equipped with the countably generated $\sigma$-algebra $\sigma (\mathcal F)$.
By Kuratowski's theorem (see e.g.~\cite[p.\ 487f]{Kur66} or \cite[Section I.3,  p.\ 15ff]{Par67}) it follows that $id^{-1}$ is $\sigma (\mathcal F) / \mathcal B (\mathbb R^\infty)$-measurable, which implies that $\mathcal B (\mathbb R^\infty) \subseteq \sigma (\mathcal F)$.
\end{proof}

\section*{Acknowledgement}
\noindent
The author would like to thank Michael Röckner and Vladimir I.~Bogachev for fruitful discussions and helpful comments.
Financial support by the DFG through the CRC 1283 ``Taming uncertainty and profiting from randomness and low regularity in analysis, stochastics and their applications'' is acknowledged.

\end{document}